\documentclass[a4]
{article}
\usepackage{graphicx}
\graphicspath{ {./images/} }
\usepackage{color}
\usepackage[T1]{fontenc}
\usepackage[french,english]{babel}
\usepackage{amssymb, amsmath}
\usepackage{amsthm}
\usepackage{tikz}
\usepackage{hyperref}
\usepackage[hmargin=2cm, vmargin=2cm]{geometry}
\def\IR{{\mathbb R}}

\def\^#1{\if#1i{\accent"5E\i}\else{\accent"5E #1}\fi}
\def\"#1{\if#1i{\accent"7F\i}\else{\accent"7F #1}\fi}

\def\paragraph#1{{\mathbf #1\ }}

 \def\RR{\mbox{\rm I~\hspace{-1.15ex}R} }
 
 \def\CC{\rm \hbox{C\kern-.56em\raise.4ex
          \hbox{$\scriptscriptstyle |$}\kern+0.5 em }}
\def\CC{\mathrm \hbox{C\hspace{-1.em}\raisebox{.47ex}{
         \textrm{\mbox{$\scriptscriptstyle |$}}}\hspace{+0.8 em} }}
\def\CC{\hskip.2em\raisebox{.22em}{\makebox[0pt][c]{$\scriptscriptstyle
|$}} \hskip-0.25em\mathrm{C}}

\def\OO{\rm \hbox{O\kern-.34em\raise.47ex
         \hbox{$\scriptscriptstyle |$}\kern-.46em\raise.47ex
         \hbox{$\scriptscriptstyle |$}\kern+0.5 em }}
\def\RR{\mbox{\mathrm I\hspace{-0.50ex}R} }

                      \catcode`@=11
\def\magnification{\afterassignment\m@g\count@}

\def\m@g{\mag\count@
  \hsize6.5truein\vsize8.9truein\dimen\footins8truein}
                      \catcode`@=12

\def\hcboxcm#1#2{\hbox to #1{\hfill #2 \hfill}}

\def\null{\hbox{}}

\def\tn1{\overline n_1}
\def\tn2{\overline n_2}
\def\tn{\overline n }

\let\ds\displaystyle

\def\ba{\begin{array}}
\def\ea{\end{array}}
\def\be{\begin{equation}}
\def\ee{\end{equation}}
\def\bea{\begin{eqnarray}}
\def\eea{\end{eqnarray}}
\def\bean{\begin{eqnarray*}}
\def\eean{\end{eqnarray*}}

\def\RR{{\mathrm{ I~\hspace{-1.15ex}R}}}

\def\qquad{{\quad\quad}}

\def\={\, = \, }

 \def\today{\number\day\space \ifcase \month\or janvier \or f\'evrier
 \or mars \or avril \or mai \or juin \or juillet \or ao\^ut \or
 septembre \or octobre \or novembre \or d\'ecembre\fi
 \number\year}

\def\Box{\leavevmode\vbox{\hrule
     \hbox{\vrule\kern4pt\vbox{\kern4pt}%
           \vrule}\hrule}}
\def\blackbox{\leavevmode\vrule height 5pt width 4pt depth 0pt\relax}
\newcounter{appendix}
\newcounter{sectionz}
\def\appendix{\advance\c@appendix by 1
\def\thesectionz {\Alph{appendix}}
\def\thesection{Appendix \Alph{appendix}:}
\ifnum\c@appendix=1 \setcounter{section}{-1} \fi
\@startsection {section}{1}{\z@}{-3.5ex plus -1ex minus
  -.2ex}{2.3ex plus .2ex}{\large\mathbf}}

\catcode`@=12
\newtheorem{lemme}{Lemma}[section]
\newtheorem{lemma}[lemme]{Lemma}
\newtheorem{theorem}[lemme]{Theorem}

\newtheorem{corollary}[lemme]{Corollary}

\newtheorem{definition}[lemme]{Definition}

\newtheorem{proposition}[lemme]{Proposition}

\newtheorem{remark}[lemme]{Remark}

\def\deblem{\begin{lemme}\it }
\def\finlem{\end{lemme}}
\def\debthm{\begin{theorem}\it }
\def\finthm{\end{theorem}}
\def\debprop{\begin{proposition} \it}
\def\finprop{\end{proposition}}
\def\debcor{\begin{corollary}\it }
\def\fincor{\end{corollary}}
\def\debdef{\begin{definition}\it}
\def\findef{\end{definition}}
\def\debrem{\begin{remark}\em}
\def\finrem{\null\hfill\blackbox\end{remark}}
\def\phi{\phi^\varepsilon}

\def\rho{\varrho}

\def\ph-Poisson systemialpha{\phi^{\alpha}}

\def\baromega{\overline{\omega}}

\begin{document}

\title{ \sc Influence of the competition in the spatial dynamics of a population of \textit{Aedes} mosquitoes }

\author{ {\small{\sc 
 SAMIA BEN ALI \thanks{Department of Mathematics, Faculty of Sciences of Tunis, University of Tunis El-Manar,  2092 El-Manar, Tunisia (samia.benali@fst.utm.tn)}
 \& MOHAMED LAZHAR TAYEB \thanks{Department of Mathematics, Faculty of Sciences of Tunis, University of Tunis El-Manar,  2092 El-Manar, Tunisia (mohamedlazhar.tayeb@fst.utm.tn)}
 \& NICOLAS VAUCHELET\thanks{Universit\'e Sorbonne Paris Nord, Laboratoire Analyse, G\'eom\'etrie et Applications, LAGA, CNRS UMR 7539,
F-93430, Villetaneuse, France (vauchelet@math.univ-paris13.fr) }}}}
\date{}
\maketitle 

\begin{abstract}
In this article, we investigate a competitive reaction-diffusion system modelling the interaction between several species of mosquitoes.
In particular, it has been observed that in tropical regions,
\textit{Aedes} aegypti mosquitoes are well established in urban area whereas \textit{Aedes} albopictus mosquitoes spread widely in forest region. 
The aim of this paper is to propose a simple mathematical system modeling this segregation phenomenon. 
More precisely, after modeling the dynamics by a competitive reaction-diffusion system with spatial heterogeneity, we prove that when there is a strong competition between these two species of mosquitoes, solutions to this system converge in long time to segregated stationary solutions.
Then we study the influence of this strong competition on the success of a population replacement strategy using \textit{Wolbachia} bacteria.
Our theoretical results are also illustrated by some numerical simulations.
\end{abstract}
\vskip1cm

\thispagestyle{empty}



\section{Introduction}

Aedes mosquitoes are one of the most invasive species of mosquitoes \cite{Benedict}, they are also responsible of the transmission of several diseases, among them dengue, yellow fever, chikungunya, zika.
Originally from tropical regions in Asia, they are now well-established in many countries, including temperated regions in Europa \cite{ECDC}, America \cite{USA,SA}, Africa \cite{Africa}, which raises fears of an outbreak in cases of the above-mentioned diseases.
Among \textit{Aedes} mosquitoes, the most common are \textit{Aedes albopictus} and \textit{Aedes aegypti}. 
It is well established that inter-specific competition exists between \textit{Aedes aegypti} and \textit{Aedes albopictus} \cite{Competition1}.
Moreover it has been observed that the geographical distribution of these two species of mosquitoes may be quite different~: \textit{Aedes aegypti} are mainly present in urban area such as city centre whereas \textit{Aedes albopictus} are most widespread in suburb area or in forested area, see e.g. \cite{GeoDistrib, Congo}.


The aim of this work is to propose and study a mathematical model of the competitive interaction between these two species that highlights their different habitats. 
A mathematical model of the competition between \textit{Aedes aegypti} and \textit{Aedes albopictus} has been recently proposed in \cite{Competition} to study the invasion of these species. To do so the authors introduce and study an advection-reaction-diffusion system with free boundaries.
In our paper, our objective is different, we want to provide a mathematical model for which we observe a segregation between the two species with respect to their habitats as it has been observed in the field, as mentioned above. 
From a mathematical point of view, segregation has been proved in reaction-diffusion models with strong competition, we refer in particular to the pioneer works \cite{Dancer1,Dancer2}.
The strong competition asymptotic allows to reduce the dimensionality of the system and to obtain explicit relation to guarantee invasion, see e.g. \cite{Girardin,GirardinReview}.
We will follow this approach and consider a reaction-diffusion system with strong competition in one space dimension.
More precisely, let us denote $n_1(t,x)$ and $n_2(t,x)$, with $t>0$ and $x\in \IR$, the densities of two species in competition. We denote by $b_1$, resp. $b_2$, the intrinsic birth rate of species $1$, resp. species $2$, and $d_1$, resp. $d_2$, the death rate of species $1$, resp. species $2$. The carrying capacities are denoted respectively by $K_1$ and $K_2$ and are assumed to depend on the spatial variable $x$ to emphasize the different habitats. The competition parameter is denoted $c$. Thus, the model reads, on $(0,+\infty)\times \RR$,
\begin{subequations}\label{syst12}
\begin{equation}\label{syst12:eq1}
\partial_t n_1 - D\partial_{xx} n_1 = b_1 n_1\left(1-\dfrac{n_1}{K_1(x)}\right) - c n_1 n_2 - d_1 n_1,
\end{equation}
\begin{equation}\label{syst12:eq2}
\partial_t n_2 - D\partial_{xx} n_2 = b_2 n_2\left(1-\dfrac{n_2}{K_2(x)}\right) - c n_1 n_2 - d_2 n_2.
\end{equation}
\end{subequations}
Here $D$ stands for the diffusion coefficient assumed to be constant.
This system is complemented with some initial data $n_1(t=0)=n_1^0$, $n_2(t=0,x)=n_2^0$.
Up to a scaling in space, we may assume that the diffusion coefficient is $D=1$. Then in the rest of the paper, we will always consider that $D=1$.

Many mathematical works have been devoted to the study of propagation phenomena in reaction-diffusion equations with space-dependent coefficients in reaction term, in particular with periodic environment, see e.g. \cite{SKT, BHN, Nadin} and references therein, and for more general non-homogeneous environment see \cite{BerestyckiNadin}.
In this paper, we do not consider a periodic environment but we consider the case of two separated regions, one modelling an urban area and another one modelling a wild domain. 
We first show that assuming that the competition is strong ($c\gg 1$), system \eqref{syst12} may be reduce to a simple scalar equation (see Theorem \ref{6.1} below). Then, we analyze the resulting equation and obtain a careful description of the long time behavior of its solutions. In particular, we get an explicit condition on the coefficients for which the two species are installed in two distinct geographical regions (see Theorem \ref{BB} below).
Therefore this model allows to recover the different spatial distribution of the two species as observed in the field.

Then, a second aim of this work is to study the influence of this competition in the replacement strategy using the endosymbiotic bacterium \textit{Wolbachia}.
Indeed it has been observed that, when carrying this bacterium, \textit{Aedes} mosquitoes are not able to transmit viruses like dengue, zika, chikungunya. Moreover, this bacterium is naturally transmitted from mother to off-springs and is characterized by a so-called cytoplasmic incompatibility, i.e. The mating between a male carrying \textit{Wolbachia} and a wild female will not produce viable eggs \cite{Curtis}. 
Then, massive releases of \textit{Wolbachia}-infected mosquitoes in the field are considered as a possible method to replace wild mosquitoes and prevent dengue epidemics \cite{MR5139942} (see also \cite{WMP}).

From a mathematical point of view, spatial propagation of this replacement technique has been studied using some reaction-diffusion systems, see e.g. \cite{BartonTurelli, Fenton, ChanKim, Hughes, Strugarek, Zubelli}. Inspired by the model in \cite{Strugarek}, we consider the following system modelling the competition of a species represented by its density $n_1$ with another species which is divided into two densities denoted $n_2$ and $n_3$, $n_2$ representing the wild mosquitoes whereas $n_3$ being the density of mosquitoes carrying \textit{Wolbachia}.  
The model reads, for $t>0$, and $x\in\IR$~:
\begin{align}\label{B-samia}
 \left\lbrace
\begin{array}{lcl}
\displaystyle \partial_t n_1 - \partial_{xx} n_1 = b_{1} n_{1}\left(1- \frac{n_{1}}{K_{1}(x)}\right) - c n_{1} (n_{2}+n_{3}) -d_{1} n_{1}, \\[4mm]
\displaystyle \partial_t n_2 - \partial_{xx} n_2 = b_{2} n_{2} \frac{n_{2}}{n_{2}+ n_{3}} \left(1-\frac{n_{2}+ n_{3} }{K_2(x)}\right)- c n_{1} n_{2}-d_{2} n_{2}, \\[4mm]
\displaystyle \partial_t n_3 - \partial_{xx} n_3 = b_{3} n_{3} \left(1-\frac{n_{2}+ n_{3} }{K_2(x)}\right)- c n_{3} n_{1}-d_{3} n_{3}, \\[4mm]
n_1(0,x) = n_1^0(x), \quad n_2(0,x) = n_2^0(x), \quad n_3(0,x) = n_3^0(x).
 \end{array}
\right. 
\end{align}
As for the above system, the parameters $\left(b_{i}, d_{i}\right),$ for $i \in\{1,2,3\}$ represent, respectively, the intrinsic birthrates and mortality, and $K_{i}$ for $i \in\{1,2\}$ denotes the environmental carrying capacity.
The factor $\frac{n_2}{n_2+n_3}$ represents the probability to mate with a partner not carrying \textit{Wolbachia}, indeed due to the cytoplasmic incompatibility only mating between uninfected mosquitoes give birth to uninfected mosquitoes.
Finally, $c$ is a competition parameter.
Then we propose a formal analysis of this system to study the success of the replacement strategy in case of a strong competition between several species of mosquitoes.

The outline of this paper is the following. 
In the next section, we first consider briefly the homogeneous environment and study the equilibria of system \eqref{syst12} and their stability. Section \ref{sect3} is devoted to the analysis of system \ref{syst12} in the heterogeneous environment, more precisely we assume that the carrying capacities $K_1$ and $K_2$ take two different values to model the two different region $\{x<0\}$ and $\{x>0\}$. We first reduce the system in the strong competition setting. Then we study the large time asymptotics of the solutions of the reduced system and provide some explicit conditions on the coefficients to guarantee a stable spatial segregation (see Theorem \ref{BB}) or to prevent one speices from invading another (see Theorem \ref{TH2}). These theoretical results are illustrated by numerical simulations which are presented in section \ref{sec:num}.
Section \ref{sec:Wol} is devoted to the investigation of system \ref{B-samia} where a species carrying Wolbachia is introduced. We first provide an analysis in the homogeneous environment of the equilibria and their stability. Then, we carry out a formal analysis of this system in the heterogeneous environment in the strong competition setting. Finally, this formal analysis is illustrated by some numerical simulations in section \ref{sec:Wolnum}.

\section{Homogeneous environment~: Equilibria and stability}

In this Section, we begin by considering the homogeneous case, then the carrying capacities $K_1$ and $K_2$ are constants and we assume also that the initial data are non-negative and do not depend on the space variable $x$. Then, we may assume that $n_1$ and $n_2$ does not depend on the space variable also and consider the simplified version of system \eqref{syst12} 
\begin{equation}\label{EDO12}
    \left\{
    \begin{array}{l}
        \displaystyle n_1' = b_1 n_1 \left(1-\frac{n_1}{K_1}\right) - c n_1 n_2 - d_1 n_1,  \\
        \displaystyle n_2' = b_2 n_2 \left(1-\frac{n_2}{K_2}\right) - c n_1 n_2 - d_2 n_2.
    \end{array}
    \right.    
\end{equation}
This system is complemented by non-negative initial data. 
By a direct application of the Cauchy-Lipschitz theorem, there exists a unique global solution of the Cauchy problem related to system \eqref{EDO12}. Moreover this solution is non-negative.
Notice that if $n_2=0$, then the species 1 is viable if $b_1>d_1$. Therefore, we will always assume
\begin{equation}
    \label{hyp:bd}
    b_1>d_1, \qquad b_2>d_2.
\end{equation}
The Jacobian associated to the right-hand side of the system \eqref{EDO12} reads
$$
\begin{aligned}
&\operatorname{Jac}(n_1,n_2)=\left(\begin{array}{cc}
b_1 \left(1-\frac{2 n_1}{K_1}\right)- c n_2-d_1  & -c n_1 \\ \\
- c n_2 & b_2 \left(1-\frac{2 n_2}{K_2}\right)- c n_1-d_2 
\end{array}\right)
\\
\end{aligned}
$$
It is readily seen that the extradiagonal terms are non positive. By Kamke-Muller conditions (see $\cite{MR6139941}$ ), this implies that the system is monotone with respect to the cone $\mathbb{R}_{+} \times \mathbb{R}_{-}$; in other words, it is competitive.

Then, the following Lemma provides the equilibria and their stability~:
\begin{lemma} \label{11}
Let us assume that \eqref{hyp:bd} holds.
Then system \eqref{EDO12} admits the three following non-negative steady states:
$$ 
(0,0),\ (n_{1}^{*},0), \ (0, n_{2}^{*}) ,\quad with \quad  n_{i}^{*} = K_i  \left(1 - \frac{d_i}{b_i} \right), i \in \lbrace 1,2 \rbrace . $$ 
Moreover, if 
\begin{align}\label{1022}
 c > \max\left\{\frac{b_1}{K_1} \left( \frac{b_2-d_2}{b_1-d_1}\right), \frac{b_2}{K_2} \left( \frac{b_1-d_1}{b_2-d_2}\right) \right\},
\end{align}
then there is another distinct non-negative steady state $(\overline{n_1}, \overline{n_2})$ where
\begin{align}\label{nbar}
\overline{n_1} =  \frac{c(b_2-d_2)- \frac{b_2}{K_2}(b_1-d_1)}{c^2 - \frac{b_1 b_2}{K_1 K_2}}, \qquad 
\overline{n_2} =  \frac{c(b_1-d_1)- \frac{b_1}{K_1}(b_2-d_2)}{c^2 - \frac{b_1 b_2}{K_1 K_2}}.
\end{align}
In this case, we have $(0, 0)$ and $(\overline{n_1}, \overline{n_2})$ are (locally linearly) unstable, whereas $(n_{1}^{*},0)$ and $(0, n_{2}^{*})$ are locally asymptotically stable.
\end{lemma}
\begin{proof}
Equilibria may be computed directly by solving the following system
$$
\left\{
   \begin{array}{l}
       \displaystyle 0 = b_1 \overline{n_1} \left(1-\frac{\overline{n_1}}{K_1}\right) - c \overline{n_1}\, \overline{n_2} - d_1 \overline{n_1},  \\
       \displaystyle 0 = b_2 \overline{n_2} \left(1-\frac{\overline{n_2}}{K_2}\right) - c \overline{n_1}\, \overline{n_2} - d_2 \overline{n_2}.
   \end{array}
\right.
$$
Under the condition \eqref{1022}, we find the four steady states solutions mentionned in the statement of the Lemma.
Then, computing the Jacobian at the equilibria we get
$$
\begin{aligned}
&\operatorname{Jac}\left(0, 0\right)=\left(\begin{array}{cc}
b_1-d_1 & 0 \\ \\
0 &  b_2-d_2
\end{array}\right).
\end{aligned}
$$
Clearly, when $b_1>d_1$ and $b_2>d_2$, the extinction equilibrium is unstable.
We have also
$$
\begin{aligned}
&\operatorname{Jac}\left(n_1^*, 0\right)=\left(\begin{array}{cc}
-\left(b_1-d_1\right) & -c n_{1}^{*}\\ \\
0 & b_2 - d_2 - c K_1 \left(1 - \frac{d_1}{b_1}\right)
\end{array}\right)
\end{aligned}$$
and
$$
\begin{aligned}
&\operatorname{Jac}\left(0, n_2^*\right)=\left(\begin{array}{cc}
b_1 - d_1 - c K_2 \left(1 - \frac{d_2}{b_2}\right) & 0 \\ \\
- c n_{2}^{*} & -\left(b_2-d_2\right)
\end{array}\right).
\end{aligned}
$$
The linear stability of $\left(n_1^*, 0\right)$ and $\left(0, n_2^*\right)$ follows from assumption \eqref{1022}.
Finally, we compute
$$
\begin{aligned}
&\operatorname{Jac}(\overline{n_1}, \overline{n_2})
=\left(\begin{array}{cc}
b_1 \left(1-\frac{2 \overline{n_1}}{K_1}\right)- c \overline{n_2}-d_1  & -c \overline{n_1} \\ \\
- c \overline{n_2} & b_2 \left(1-\frac{2 \overline{n_2}}{K_2}\right)- c \overline{n_1}-d_2 
\end{array}\right)
=\left(\begin{array}{cc}
a_{11}  & a_{12}  \\ \\
a_{21}  & a_{22}  
\end{array}\right)
\\
\end{aligned}
$$
with $\overline{n_1}$ and $\overline{n_2}$ defined in \eqref{nbar}. 
Then, we can deduce that 
$$
\frac{2 b_1}{K_1} \overline{n_1} + c \overline{n_2} =\ds{\frac{1}{\Delta} \lbrace 2 c \frac{b_1}{K_1} (b_2 - d_2) - 2 \frac{b_1 b_2}{K_1 K_2}(b_1 - d_1) + c^2 (b_1 - d_1) - c \frac{b_1}{K_1}(b_2 - d_2) \rbrace}, 
$$
where $\Delta = c^2 - \frac{b_1 b_2}{K_1 K_2}$,
and we calculate $a_{11}$, we get
\begin{align*}
a_{11} = &  \frac{1}{\Delta}\bigg(\Delta (b_1 - d_1) - 2 c \frac{b_1}{K_1} (b_2 - d_2) + 2 \frac{b_1 b_2}{K_1 K_2}(b_1 - d_1)-  c^2 (b_1 - d_1) +  c \frac{b_1}{K_1}(b_2 - d_2)\bigg) \nonumber\\
= & \frac{1}{\Delta} \bigg( c^2 (b_1 - d_1) - \frac{b_1 b_2}{K_1 K_2}(b_1 - d_1) - 2 c \frac{b_1}{K_1} (b_2 - d_2) - c^2 (b_1 - d_1) + c \frac{b_1}{K_1} (b_2 - d_2)\bigg) \nonumber \\
= &  \frac{1}{\Delta} \frac{b_1}{K_1} \bigg( \frac{ b_2}{ K_2} (b_1 - d_1) - c (b_2 - d_2) \bigg).
\end{align*}
Similarly, we obtained
\begin{align*}
a_{22} = \frac{1}{\Delta} \frac{b_2}{K_2}\bigg( \frac{b_1}{ K_1} (b_2 - d_2) - c  (b_1 - d_1) \bigg).
\end{align*}
Then,
\begin{align*}
\operatorname{det}\left(\operatorname{Jac}(\overline{n_1}, \overline{n_2})\right) = &   \frac{b_1 b_2}{K_1 K_2} \frac{1}{\Delta^2}  \bigg( \frac{b_2}{K_2} (b_1 - d_1) - c (b_2 - d_2)\bigg)  \bigg( \frac{b_1}{K_1} (b_2 - d_2) - c (b_1 - d_1)\bigg) \nonumber \\
  - & \frac{c^2}{\Delta^2} \bigg(c (b_2 - d_2) - \frac{b_2}{K_2} (b_1 - d_1) \bigg)  \bigg(c (b_1 - d_1) - \frac{b_1}{K_1} (b_2 - d_2) \bigg)\nonumber \\
 = & \frac{1}{\Delta^2} \bigg( c (b_2 - d_2) - \frac{b_2}{K_2} (b_1 - d_1) \bigg)  \bigg( c (b_1 - d_1) - \frac{b_1}{K_1} (b_2 - d_2) \bigg) 
  \bigg(  \frac{b_1 b_2}{K_1 K_2} - c^2 \bigg). 
\end{align*}
 Then  according to the condition \eqref{1022}, which clearly implies $c^2 > \frac{b_1 b_2}{K_1 K_2}$, we have $\operatorname{det}\left(\operatorname{Jac}(\overline{n_1}, \overline{n_2})\right) < 0$. It implies the unstability of $(\overline{n_1}, \overline{n_2})$.
\end{proof}

\section{ Analysis in a heterogeneous spatial environment }\label{sect3}

In this section, we consider that there are heterogeneities in the environment; therefore the carrying capacity depends on the space variable. Moreover, to be able to have a precise description of the dynamics, we first reduce the system by assuming that the competition is strong.

\subsection{Reduction in a strong competition setting}\label{sec:reduction}

Let us assume that the carrying capacities verify, for $i=1,2$,
\begin{equation}
    \label{hyp:K1K2}
     K_i \in L^\infty(\RR), \qquad \frac{1}{K_i} \in L^\infty(\RR) \cap BV(\RR).
\end{equation}
 In order to simplify the considered system, we place ourselves in a strong competition setting.
  More precisely, we consider that the competition parameter $c$ is large, we fix it to be $c=\frac{1}{\varepsilon}$ with a parameter $0<\varepsilon\ll 1$ and we investigate the asymptotic $\varepsilon\to 0$. 
 The system reads for $t>0$, and $x \in \mathbb{R}$, 
\begin{align}\label{Da-ssamia}
 \left\lbrace 
\begin{array}{lcl}
\partial_t n_{1,\varepsilon} - \partial_{xx} n_{1,\varepsilon} = b_1 n_{1,\varepsilon} \left(1-\dfrac{n_{1,\varepsilon}}{K_1(x)}\right)- \dfrac{1}{\varepsilon} n_{1,\varepsilon} n_{2,\varepsilon} - d_1 n_{1,\varepsilon}, \\
\\
\partial_t n_{2,\varepsilon} - \partial_{xx} n_{2,\varepsilon} = b_2 n_{2,\varepsilon} \left(1-\dfrac{n_{2,\varepsilon}}{K_2(x)}\right)- \dfrac{1}{\varepsilon} n_{1,\varepsilon} n_{2,\varepsilon} - d_2 n_{2,\varepsilon}.
 \end{array}
\right. 
\end{align}
This system is complemented by initial data
$$
n_{1,\varepsilon}(t=0,x)=n_{1,\varepsilon}^0(x),\quad
n_{2,\varepsilon}(t=0,x) = n^0_{2,\varepsilon}(x).
$$
We assume that these initial data are nonnegative, continuous on $\RR$ and uniformly bounded in $W^{1,1}(\mathbb{R})\cap L^\infty(\mathbb{R})$~:
\begin{align} \label{54}
\left\|n_{1,\varepsilon}^0\right\|_{L^1\left(\mathbb{R}\right)}+\left\|n_{1,\varepsilon}^0\right\|_{L^{\infty}\left(\mathbb{R}\right)}+\left\|n_{2,\varepsilon}^0\right\|_{L^1 \left(\mathbb{R}\right)}+\left\|n_{2,\varepsilon}^0\right\|_{L^{\infty}\left(\mathbb{R}\right)} \leq C^0,
\end{align}
\begin{align} \label{53}
\left\|\partial_x n_{1,\varepsilon}^0\right\|_{L^1\left(\mathbb{R} \right)}+\left\|\partial_x n_{2,\varepsilon}^0\right\|_{L^1\left(\mathbb{R} \right)} \leq C^1. 
\end{align}
From these assumptions, we may extract a subsequence still denoted by ($n_{1,\varepsilon}^0, n_{2,\varepsilon}^0)$ of initial data that convergences strongly in $L^1(\mathbb{R}) \times L^1(\mathbb{R})$. We assume moreover that the limiting initial data are segregated, i.e.
\begin{equation}\label{assumpini}
n_{1,\varepsilon}^0 \rightarrow n_1^0, \quad n_{2,\varepsilon}^0 \rightarrow n_2^0 \quad \text {  in } L^1\left(\mathbb{R}\right), \text{ and }
n_1^0 n_2^0 = 0 .
\end{equation}
Following the ideas in \cite{Dancer1,Girardin,Perthame}, we obtain the following result~:
\begin{theorem}\label{6.1}
Let us assume that \eqref{hyp:K1K2} holds. Let us consider system \eqref{Da-ssamia} with assumptions \eqref{54}-\eqref{53} on the initial data. Then, as $\varepsilon\to 0$, we have
$$
n_{1,\varepsilon} \rightarrow n_1, \quad n_{2,\varepsilon} \rightarrow n_2 \quad \text { strongly in } L_{\mathrm{loc}}^1\left(\mathbb{R}^{+}, L^1(\mathbb{R})\right),
$$
where $n_1, n_2 \in L^{\infty}\left(\mathbb{R}^{+} ; L^1 \cap L^{\infty}\left(\mathbb{R} \right)\right)$,  $n_1 n_2(t, x)=0$  a.e. and $\omega =n_1-n_2$ solves 
\begin{equation}\label{eq:w}
\partial_t \omega - \partial_{xx} \omega = b_1 \omega_+ \left(1-\dfrac{\omega_+}{K_1(x)}\right)-d_1 \omega_+ - b_2 \omega_-\left(1-\dfrac{\omega_-}{K_2(x)}\right) + d_2 \omega_-,
\end{equation}
where we use the standard notations for the positive and negative parts, i.e. $u_+ = \max\{0,u\}$, and $u_- = \max\{0,-u\}$.
This equation is complemented with initial data $\omega^0(x)= n_{1}^0(x) -  n_{2}^0(x).$
\end{theorem}
\begin{proof}
We split the proof into several steps.

\textbf{ Step 1: Uniform a priori estimates  } \\
First, it is classical to show that the solutions to system \eqref{Da-ssamia} are nonnegative, when initial data are nonnegative. Moreover, by parabolic regularity, we have $n_{1,\varepsilon}$ and $n_{2,\varepsilon}$ belongs to $L^\infty(\IR^+,C^1(\IR))$. Then, applying the maximum principle, we obtain the $L^\infty$-bounds (uniform with respect to $\varepsilon$)
\begin{align}
0 \leq n_{1,\varepsilon}(t, x) \leq \max\left\{\left\|n_{1,\varepsilon}^0\right\|_{L^{\infty}\left(\mathbb{R}\right)}, \left\|K_1\right\|_{L^{\infty}\left(\mathbb{R}\right)}\right\}, \quad 0 \leq n_{2,\varepsilon}(t, x) \leq \max\left\{\left\|n_{2,\varepsilon}^0\right\|_{L^{\infty}\left(\mathbb{R}\right)}, \left\|K_2\right\|_{L^{\infty}\left(\mathbb{R}\right)}\right\}.
\end{align}
Integrating the system \eqref{Da-ssamia} we get the $L^1$ estimate
\begin{align} \label{59}
 \left\{\begin{array}{l}
\ds\int_{\mathbb{R}} n_{1,\varepsilon}(t, x) d x+ \int_0^t \int_{\mathbb{R}} \frac{n_{1,\varepsilon}(s, x) n_{2,\varepsilon}(s, x)}{\varepsilon} \,dx ds  \\[3mm]
\qquad +  \ds \int_0^t \int_{\mathbb{R}} \big( \frac{n_{1,\varepsilon}(s, x)^2 }{K_1(x)} + (d_1-b_1)n_{1,\varepsilon} \big)\,dxds 
\leq  \ds\int_{\mathbb{R}} n_{1,\varepsilon}^0(x) d x, \\[4mm] 
\ds\int_{\mathbb{R}} n_{2,\varepsilon}(t, x) d x+\int_0^t \int_{\mathbb{R}} \frac{n_{1,\varepsilon}(s, x) n_{2,\varepsilon}(s, x)}{\varepsilon} dx ds \\[3mm]
\qquad + \ds \int_0^t \int_{\mathbb{R}} \big( \frac{n_{2,\varepsilon}(s, x)^2 }{K_2(x)} + (d_2-b_2)n_{2,\varepsilon} \big) \,dxds 
 \leq \int_{\mathbb{R}} n_{2,\varepsilon}^0(x) \,dx .
\end{array}\right.
\end{align}
Next, we now prove the  estimates in $x$. We work on the equations of \eqref{Da-ssamia} and
differentiate them with respect to $x$. We multiply by the sign function and we sum and integrate to obtain
\begin{align*}
& \frac{d}{d t} \int_{\mathbb{R}}\left[\left|\frac{\partial}{\partial x} n_{1,\varepsilon}\right|+\left|\frac{\partial}{\partial x} n_{2,\varepsilon}\right|\right] \,dx   \\
 \quad & \leq  - \frac{1}{\varepsilon}\int_{\mathbb{R}}\left[\frac{\partial}{\partial x} n_{1,\varepsilon} n_{2,\varepsilon}+ n_{1,\varepsilon} \frac{\partial}{\partial x} n_{2,\varepsilon}\right]\left[\operatorname{sgn}\left(\frac{\partial}{\partial x} n_{1,\varepsilon}\right)+\operatorname{sgn}\left(\frac{\partial}{\partial x} n_{2,\varepsilon}\right)\right]\,dx \\
&\quad +  C\int_{\mathbb{R}} \left( \left|\frac{\partial}{\partial x} n_{1,\varepsilon}\right| +   \left|\frac{\partial}{\partial x} n_{2,\varepsilon}\right| +  \left|\frac{\partial}{\partial x} \frac{n_{1,\varepsilon}(s, x)^2 }{K_1(x)} \right | + \left|\frac{\partial}{\partial x} \frac{n_{2,\varepsilon}(s, x)^2 }{K_2(x)} \right | \right).
\end{align*}
Notice that due to the regularity of $n_{1,\varepsilon}$ and $n_{2,\varepsilon}$, all terms in the above computations are well-defined. Then, from standard computations, we have
\begin{align*}
& \left[\frac{\partial}{\partial x} n_{1,\varepsilon} n_{2,\varepsilon}+ n_{1,\varepsilon} \frac{\partial}{\partial x} n_{2,\varepsilon}\right]\left[\operatorname{sgn}\left(\frac{\partial}{\partial x} n_{1,\varepsilon}\right)+\operatorname{sgn}\left(\frac{\partial}{\partial x} n_{2,\varepsilon}\right)\right]  \\
& =  \left(\left|\frac{\partial}{\partial x} n_{1,\varepsilon}\right| + \frac{\partial}{\partial x} n_{1,\varepsilon} \operatorname{sgn}\left(\frac{\partial}{\partial x} n_{2,\varepsilon}\right)\right) n_{2,\varepsilon}
+\left(\left|\frac{\partial}{\partial x} n_{2,\varepsilon}\right|  
 +\frac{\partial}{\partial x} n_{2,\varepsilon} \operatorname{sgn}\left(\frac{\partial}{\partial x} n_{1,\varepsilon}\right)\right) n_{1,\varepsilon} \\
& \geq  0.
\end{align*}
Using assumption \eqref{hyp:K1K2}, the regularity, and the $L^1\cap L^\infty$ uniform estimate on $n_{1,\varepsilon}$ and $n_{2,\varepsilon}$, we deduce the following estimate
$$
\frac{d}{d t} \int_{\mathbb{R}}\left[\left|\frac{\partial}{\partial x} n_{1,\varepsilon}\right|+\left|\frac{\partial}{\partial x} n_{2,\varepsilon}\right|\right] \,dx \leq C_1 \int_{\mathbb{R}}\left[\left|\frac{\partial}{\partial x} n_{1,\varepsilon}\right|+\left|\frac{\partial}{\partial x} n_{2,\varepsilon}\right|\right]\,dx + C_2. 
$$
We conclude by a Gronwall argument.

\textbf{ Step 2: Compactness   } \\
 It remains to show the strong
convergence of $n_{1,\varepsilon}$ and $n_{2,\varepsilon}$. This follows from the a priori estimates, which imply
local compactness, and the time compactness may be obtained from the Lions-Aubin lemma \cite{Simon}. 
More precisely, rewriting the system \eqref{Da-ssamia} as
\begin{align}\label{729}
\frac{\partial}{\partial t} n_{1,\varepsilon}=\partial_{xx}n_{1,\varepsilon}+ f^1_{\varepsilon}(t, x), \quad \frac{\partial}{\partial t} n_{2,\varepsilon}=\partial_{xx}n_{2,\varepsilon} + f^2_{\varepsilon}(t, x).
\end{align}
From above estimates, the quantities $f^1_{\varepsilon}(t,x)$, $f^2_{\varepsilon}(t,x)$ are uniformly bounded in $L_{t, x}^1$.
Moreover, we obtaine the uniform bound
$$
\left\|\partial_x n_{1,\varepsilon}(t)\right\|_{L^1\left(\mathbb{R}\right)}+\left\|\partial_x  n_{2,\varepsilon}(t)\right\|_{L^1\left(\mathbb{R}\right)} \leq C.
$$
We deduce that $\partial_t n_{1,\varepsilon}$ and $\partial_t n_{2,\varepsilon}$ are uniformly bounded in $L^1_{loc}((\mathbb{R}^+), W^{-1,1}(\mathbb{R}))$.
Finally, according to Lions-Aubin lemma \cite{Simon}, we get that the sequences
$(n_{1,\varepsilon})_\varepsilon$ and $(n_{2,\varepsilon})_\varepsilon$ are relatively compact in $L^1_{\text{loc}} \left(\mathbb{R}^{+} ; L^1_{\text{loc}} \left(\mathbb{R}\right)\right).$ 

To pass from local convergence in space to global convergence, we need to prove that the mass in the tail $|x|>R$ is uniformly small with respect to $\varepsilon$ when $R$ is large. Let us introduce the cutoff function $\xi\in C^\infty(\IR)$ such that $\xi(x) = 0$ for $|x|<1$ and $\xi(x)=1$ for $|x|\geq 2$ and $0\leq \xi\leq 1$. Then we deduce from \eqref{Da-ssamia}
\begin{align*}
\frac{d}{dt} \int_{\IR} n_{1,\varepsilon}(t,x)\,\xi\left(\frac{x}{R}\right)\,dx 
& \leq \int_\IR \partial_{xx}n_{1,\varepsilon}(t,x) \, \xi\left(\frac{x}{R}\right) + b_1 \int_\IR n_{1,\varepsilon}(t,x) \, \xi\left(\frac{x}{R}\right)\,dx  \\
& \leq \frac{1}{R^2} \int_\IR n_{1,\varepsilon}(t,x) \, \xi''\left(\frac{x}{R}\right)\,dx 
    + b_1 \int_{\IR} n_{1,\varepsilon}(t,x) \, \xi\left(\frac{x}{R}\right)\,dx.
\end{align*}
Then, after an integration in time, for any $t\in[0,T]$, $T>0$, we have
\begin{align*}
\int_{\IR} n_{1,\varepsilon}(t,x)\,\xi\left(\frac{x}{R}\right)\,dx 
& \leq e^{b_1 t} \left(\int_{\IR} n_{1,\varepsilon}^0(x) \, \xi\left(\frac{x}{R}\right)\,dx + \frac{CT}{R^2}\right) \\
& \leq e^{b_1 t} \left( \|n_{1,\varepsilon}^0- n_1^0\|_{L^1(\IR)} + \int_\IR n_{1}^0(x)\, \xi\left(\frac{x}{R}\right)\,dx + \frac{CT}{R^2} \right).
\end{align*}
The latter term  beening uniformly small for $\varepsilon$ small enough and $R$ large enough, it implies the control on the tail. 
We proceed in the same way for $(n_{2,\varepsilon})_\varepsilon$.

\textbf{ Step 3: Passing to the limit.} \\
From the compactness result, we may extract subsequences, still denoted $n_{1,\varepsilon}$ and $n_{2,\varepsilon}$, which converge strongly in $L^1_{\text{loc}} \left(\mathbb{R}^{+} ; L^1 \left(\mathbb{R}\right)\right)$ and almost everywhere to limits denoted by 
$$n_1, n_2 \in L^1\cap L^{\infty}\left(\mathbb{R}^{+} \times \mathbb{R}\right).$$
Passing to the limit into \eqref{59}, we deduce that 
$$
\int_0^t \int_{\IR} n_1(s,x) n_2(s,x) \,dxds = 0.
$$
It implies the segregation property $n_1 n_2 = 0$ a.e. and for $t=0$ thanks to assumption \eqref{assumpini}.
Next, we define
$$
\omega_{\varepsilon}=n_{1,\varepsilon}-n_{2,\varepsilon}
\underset{\varepsilon \rightarrow 0}{\longrightarrow} \omega=n_1-n_2, \quad and \quad \omega^0=n_1^0-n_2^0. 
$$
From the segregation property, we deduce that
$$
\omega_+  = n_1, \qquad \text{ and } \quad \omega_- = n_2.
$$
Then, subtracting the two equations in \eqref{Da-ssamia} and multiplying by a test function $\varphi(t,x)$, we obtain after an integration
\begin{align*}
  & -\int_0^T \int_{\mathbb{R}} \omega^{\varepsilon}(t,x) \partial_t \varphi(t,x) dt dx + \int_{\mathbb{R}}  \omega^{\varepsilon}(0,x) \varphi(0,x) dx  - \int_0^T \int_{\mathbb{R}} \omega^{\varepsilon}(t,x)  \partial_{xx} \varphi(t,x) \,dxdt \\ 
  & = \int_0^T \int_{\mathbb{R}} \left(b_1 n_{1,\varepsilon} \left(1-\dfrac{n_{1,\varepsilon}}{K_1(x)}\right) - d_1 n_{1,\varepsilon} - b_2 n_{2,\varepsilon} \left(1-\dfrac{n_{2,\varepsilon}}{K_2(x)}\right) + d_2 n_{2,\varepsilon} \right)\varphi(t,x) \,dxdt.  
\end{align*}
Letting $\varepsilon$ going to $0$ we get 
\begin{align*}
  & -\int_0^T \int_{\mathbb{R}} \omega(t,x) \partial_t \varphi(t,x) dt dx + \int_{\mathbb{R}} \omega(0,x)   \varphi(0,x) dx- \int_0^T \int_{\mathbb{R}} \omega(t,x) \partial_{xx} \varphi(t,x) \,dxdt \\ 
  & = \int_0^T \int_{\mathbb{R}} \left(b_1 \omega_+ \left(1-\dfrac{\omega_+}{K_1(x)}\right) - d_1 \omega_+ - b_2 \omega_- \left(1-\dfrac{\omega_-}{K_2(x)}\right) + d_2 \omega_- \right) \varphi(t,x) \,dxdt.
\end{align*}
We recover equation \eqref{eq:w}.

Finally, we claim that the whole sequence is converging since the limit equation \eqref{eq:w} admits a unique solution. Indeed, uniqueness is a classical consequence of the fact that the right hand side of \eqref{eq:w} is locally Lipschitz with respect to $\omega$ since the positive part and negative part are Lipschitz function (noticing that we may write $\omega^+=\frac{1}{2}(|\omega|+\omega)$ and $\omega^- = \frac 12(|\omega|-\omega)$).
\end{proof}


\begin{remark}\label{rem:bistability}
It is important to notice that, when the carrying capacities are constant, the function defining the right hand side of the limiting reaction-diffusion equation \eqref{eq:w} is bistable. More precisely, denoting 
\begin{align}\label{ABA}
f(\omega) = b_1 \omega_+ \left(1-\frac{\omega_+}{K_1}\right) -d_1 \omega_+ - b_2 \omega_- \left(1-\frac{\omega_- }{K_2}\right)  + d_2 \omega_-,
\end{align}
we have that $f(K_1(1-\frac{d_1}{b_1})) = f(0) = f(-K_2(1-\frac{d_2}{b_2}))=0$, $f<0$ on $(-K_2(1-\frac{d_2}{b_2}),0)$, and $f>0$ on $(0,K_1(1-\frac{d_1}{b_1}))$.
As a consequence, a well-known result for such bistable reaction-diffusion equation (see e.g. \cite{Perthame} and references therein) is that in the homogeneous setting (i.e. when $K_1$ and $K_2$ are constants), there exists a traveling wave whose velocity sign is given by the sign of the quantity 
$$
\gamma := \int_{-K_2(1-\frac{d_2}{b_2})}^{K_1(1-\frac{d_1}{b_1})} f(\omega)\,d\omega = F\left(K_1(1-\frac{d_1}{b_1})\right) - F\left(-K_2(1-\frac{d_2}{b_2})\right),
$$
where $F$ denotes an antiderivative of $f$ : If $\gamma>0$ the species $1$ is invasive, whereas if $\gamma<0$ the species $2$ is invasive.
\end{remark}

\subsection{Large time behaviour of the solutions of the reduced equation}
We have seen in Theorem \ref{6.1} that in the strong competition setting, the study of system \eqref{Da-ssamia} may be reduced to the study of a scalar problem on $\mathbb{R}$ \eqref{eq:w}.
Then, in this part, we consider this simplified equation and we study stationary solutions for this problem. We first introduce the following notations for the reaction term~:
\begin{align}\label{ccc}
 f(x,\omega)=  b_1 \omega_+ \left(1-\frac{\omega_+}{K_1(x)}\right) -d_1 \omega_+ - b_2 \omega_- \left(1-\frac{\omega_- }{K_2(x)}\right)  + d_2 \omega_-, 
\end{align}
where 
$K_1(x)= K_1^F 1_{\{x<0 \}}  + K_1^U 1_{\{x>0 \}}$ and $ K_2(x)= K_2^F 1_{\{x<0 \}}  + K_2^U 1_{\{x>0 \}}$.
As above, we will use the notations $f(x,\omega)= f^U(\omega )$ for $x>0$, and $f(x,\omega) = f^F(\omega)$ for $x<0$. We denote the antiderivatives of $f^U$ and $f^F$ respectively by $F^U$ and $F^F$, i.e. $(F^U)'=f^U$ and $(F^F)'=f^F$.

Then, we consider the dynamical equation
\begin{equation}
    \label{eq:omegat}
    \partial_t \omega - \partial_{xx} \omega = f(x,\omega),
\end{equation}
with $f$ defined in \eqref{ccc}. We complement this equation with an initial data $\omega(t=0,x) = \omega^0(x)$. 

Moreover, we assume that the species 1 is invasive in the region $\{x<0\}$, whereas the species 2 is invasive in the region $\{x>0\}$. Due to Remark \ref{rem:bistability}, it implies that we make the following assumptions~:

\begin{align}\label{BBbb-ssamia}
 \left\lbrace
\begin{array}{lcl}
\displaystyle\gamma^{F} := F^F\left(K_1^F \left(1- \frac{d_1}{b_1}\right) \right) - F^F \left(-K_2^F \left(1- \frac{d_2}{b_2}\right) \right) >0, \\[4mm]
\displaystyle\gamma^{U} :=  F^U \left(K_1^U \left(1- \frac{d_1}{b_1}\right) \right)- F^U \left(-K_2^U \left(1- \frac{d_2}{b_2}\right) \right)<0.
\end{array}
\right. 
\end{align}

As a consequence, the antiderivative $F^F$ admits a global maximum in $K_1^F(1-\frac{d_1}{b_1})$ and the antiderivative $F^U$ admits a global maximum in $-K_2^U(1-\frac{d_2}{b_2})$, see Figures \ref{fig:FF} and \ref{fig:FU}.

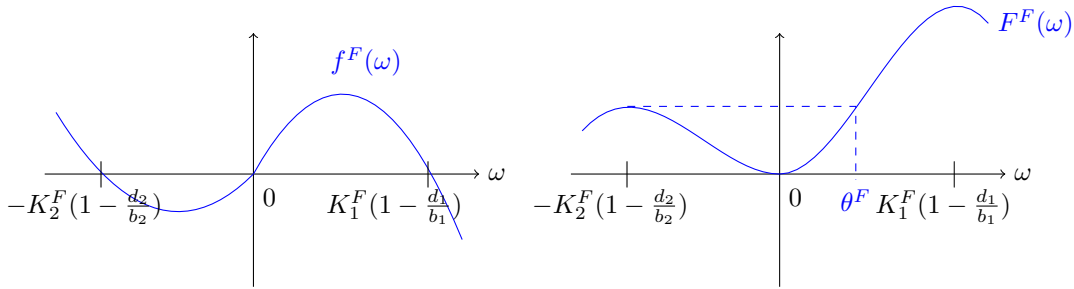
\begin{figure}[ht]
\begin{center}
\begin{tikzpicture}[scale=0.75]
    \draw[->] (-3.7,0) -- (4,0) node[right] {$\omega$};
    \draw[->] (0,-2) -- (0,2.5);
    \draw[color=blue,domain=0:3.7] plot (\x,{2*\x*(1-\x/3.5)-0.2*\x});
    \draw[color=blue,domain=-3.5:0]   plot (\x,{1.5*\x*(1+\x/4)-0.5*\x});   
    \draw[color=blue] (2,2) node {$f^F(\omega)$};
    \draw (0,-0.1) node[below right] {$0$};
    \draw (2.5,-0.1) node[below] {$K_1^F(1-\frac{d_1}{b_1})$};
    \draw (-3,-0.1) node[below] {$-K_2^F(1-\frac{d_2}{b_2})$};
    \draw (-2.7,0) node {$|$};
    \draw (3.1,0) node {$|$};
\end{tikzpicture} 
\begin{tikzpicture}[scale=0.75]
    \draw[->] (-3.7,0) -- (4,0) node[right] {$\omega$};
    \draw[->] (0,-2) -- (0,2.5);
    \draw[color=blue,domain=0:3.7] plot (\x,{0.9*\x*\x -2/3*\x*\x*\x/3.5})  node[right] {$F^F(\omega)$};
    \draw[color=blue,domain=-3.5:0]   plot (\x,{\x*\x/2+\x*\x*\x/8}) ;
    \draw[-,dashed,color=blue] (-2.7,1.2) -- (1.35,1.2) -- (1.35,-0.1) node[below] {$\theta^F$};
    \draw (0,-0.1) node[below right] {$0$};
    \draw (2.9,-0.1) node[below] {$K_1^F(1-\frac{d_1}{b_1})$};
    \draw (-3,-0.1) node[below] {$-K_2^F(1-\frac{d_2}{b_2})$};
    \draw (-2.7,0) node {$|$};
    \draw (3.1,0) node {$|$};
\end{tikzpicture} 
\caption{Schematic representations of the functions $f^F$ and $F^F$ when $\gamma^F>0$.}\label{fig:FF}
\end{center}
\end{figure}
\begin{figure}[ht]
\begin{center}
\begin{tikzpicture}[scale=0.75]
    \draw[->] (-4,0) -- (3.7,0) node[right] {$\omega$};
    \draw[->] (0,-2) -- (0,2.5);
    \draw[color=blue,domain=-3.7:0] plot (\x,{2*\x*(1+\x/3.5)-0.2*\x});
    \draw[color=blue,domain=0:3.5]   plot (\x,{1.5*\x*(1-\x/4)-0.5*\x});   
    \draw[color=blue] (2,1.1) node {$f^U(\omega)$};
    \draw (0,-0.1) node[below right] {$0$};
    \draw (2.5,-0.1) node[below] {$K_1^U(1-\frac{d_1}{b_1})$};
    \draw (-3.2,-0.1) node[below] {$-K_2^U(1-\frac{d_2}{b_2})$};
    \draw (-3.2,0) node {$|$};
    \draw (2.7,0) node {$|$};
\end{tikzpicture} 
\begin{tikzpicture}[scale=0.75]
    \draw[->] (-4,0) -- (3.7,0) node[right] {$\omega$};
    \draw[->] (0,-2) -- (0,2.5);
    \draw[color=blue,domain=-3.7:0] plot (\x,{0.9*\x*\x +2/3*\x*\x*\x/3.5});
    \draw[color=blue,domain=0:3.5]  plot (\x,{\x*\x/2-\x*\x*\x/8})   node[right] {$F^U(\omega)$};
    \draw[-,dashed,color=blue] (2.7,1.2) -- (-1.35,1.2) -- (-1.35,-0.1) node[below] {$\theta^U$};
    \draw (0,-0.1) node[below right] {$0$};
    \draw (2.5,-0.1) node[below] {$K_1^U(1-\frac{d_1}{b_1})$};
    \draw (-3.2,-0.1) node[below] {$-K_2^U(1-\frac{d_2}{b_2})$};
    \draw (-3.2,0) node {$|$};
    \draw (2.7,0) node {$|$};
\end{tikzpicture} 
\caption{Schematic representations of the functions $f^U$ and $F^U$ when $\gamma^U<0$.}\label{fig:FU}
\end{center}
\end{figure}
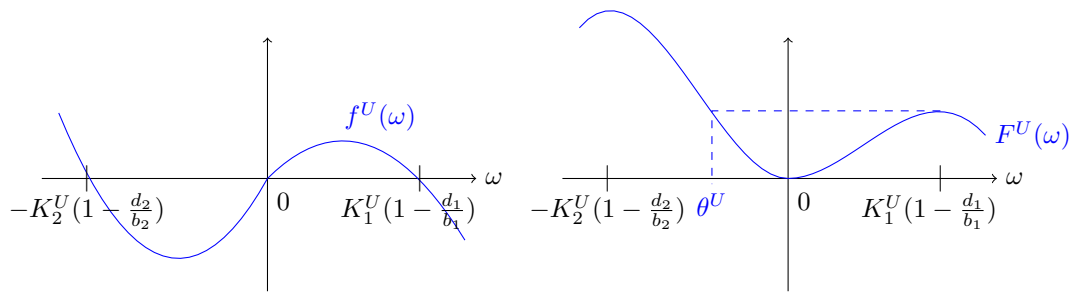

Before stating our main result, we introduce   the 
 so-called \textit{critical bubbles} (see e.g. \cite{BartonTurelli,Zubelli})~: 
\begin{lemma}
    \label{lem:bubbleF}
    Assume \eqref{hyp:bd} holds.
    Let us assume that $\gamma^F>0$, and denote $\theta^F$ the unique real in $(0,K_1^F(1-\frac{d_1}{b_1}))$ such that $F^F(\theta^F) = F^F \left(-K_2^F \left(1- \frac{d_2}{b_2}\right) \right)$ (see Figure \ref{fig:FF}).
    Then, for any $\alpha\in (\theta^F,K_1^F\left(1-\frac{d_1}{b_1}\right))$, there exists a family of \textit{bubbles}, denoted $\chi^{0,F}_\alpha$, such that $\chi_\alpha^{0,F}$ is even, nonincreasing on $(0,+\infty)$, and
    $$
    \begin{cases}
        &-(\chi_\alpha^{0,F})'' = f^F(\chi_\alpha^{0,F}) \qquad \text{ on } (-L_\alpha^F,L_\alpha^F), \\[3mm]
        &\chi_\alpha^{0,F}(0) = \alpha, \quad \chi_\alpha^{0,F} (\pm L_\alpha^F) = -\max(K_2^U,K_2^F)\left(1-\dfrac{d_2}{b_2}\right), \\[3mm]
        &\chi_\alpha^{0,F}(x) = -\max(K_2^F,K_2^U)\left(1-\dfrac{d_2}{b_2}\right) \quad \text{ on } \IR \setminus (-L_\alpha^F,L_\alpha^F),
    \end{cases}
    $$
    where
    $$
    L_\alpha^F := \int_{-\max(K_2^F,K_2^U)(1-\frac{d_2}{b_2})}^\alpha \frac{dz}{\sqrt{2(F^F(\alpha)-F^F(z))}}. 
    $$
\end{lemma}
In the same spirit, we have similarly~:
\begin{lemma}
    \label{lem:bubbleU}
    Assume \eqref{hyp:bd} holds.
    Let us assume that $\gamma^U<0$, and denote $\theta^U$ the unique real in $(-K_2^U(1-\frac{d_2}{b_2}),0)$ such that $F^U(\theta^U) = F^U \left(K_1^U \left(1- \frac{d_1}{b_1}\right) \right)$ (see Figure \ref{fig:FU}).
    Then, for any $\beta\in (-K_2^U\left(1-\frac{d_2}{b_2}\right),\theta^U)$, there exists a family of \textit{bubbles}, denoted $\chi^{0,U}_\beta$, such that $\chi_\beta^{0,U}$ is even, nondecreasing on $(0,+\infty)$, and
    $$
    \begin{cases}
        &-(\chi_\beta^{0,U})'' = f^U(\chi_\beta^{0,U}) \qquad \text{ on } (-L_\beta^U,L_\beta^U), \\[3mm]
        &\chi_\beta^{0,U}(0) = \beta, \quad \chi_\beta^{0,U} (\pm L_\beta^U) = \max(K_1^U,K_1^F)\left(1-\dfrac{d_1}{b_1}\right), \\[3mm]
        &\chi_\beta^{0,U}(x) = \max(K_1^U,K_1^F)\left(1-\dfrac{d_1}{b_1}\right) \quad \text{ on } \IR \setminus (-L_\beta^U,L_\beta^U),
    \end{cases}
    $$
    where
    $$
    L_\beta^U := \int_{\beta}^{\max(K_1^U,K_1^F)(1-\frac{d_1}{b_1})} \frac{dz}{\sqrt{2(F^U(\beta)-F^U(z))}}. 
    $$
\end{lemma}
We are now in position to state the main result of this section~:
\begin{theorem}\label{BB}
With the above notations, let us assume that \eqref{hyp:bd} and \eqref{BBbb-ssamia} hold.
We assume moreover that the initial data $\omega^0$ is such that, there exist $\alpha\in (\theta^F,K_1^F(1-\frac{d_1}{b_1}))$, $x_\alpha > L_\alpha^F$, and $\beta\in (-K_2^U(1-\frac{d_2}{b_2}),\theta^U)$, $y_\beta>L_\beta^U$ such that, for all $x\in\IR$,
\begin{equation}
    \label{eq:omegaini}
    \chi_\alpha^{0,F}(x+x_\alpha) \leq \omega^0(x) \leq \chi_\beta^{0,U}(x-y_\beta).
\end{equation}
Then, the solution of the equation \eqref{eq:omegat} converges as $t$ goes to $+\infty$, for any $x\in\IR$, towards the unique stationary solution $\baromega$ of the problem
\begin{align}\label{BB-ssamia}
 \left\lbrace
\begin{array}{lcl}
 - \baromega'' = f(x,\baromega), \quad \text{ on } \IR, \\[2mm]
  \displaystyle\baromega(-\infty) = K_1^{F}\left(1 - \frac{d_1}{b_1}\right), \qquad 
  \displaystyle\baromega(+ \infty)  = -K_2^{U}\left(1 - \frac{d_2}{b_2}\right).
 \end{array}
\right. 
\end{align}
Moreover, $\baromega$ is decreasing on $\mathbb{R}$.
\end{theorem}
The interpretation of this Theorem is that, if the region $\{x<0\}$ is favourable to species 1 whereas the region $\{x>0\}$ is favourable to species 2, and if there are enough members of species 1 in the region $\{x<0\}$ and enough members of species 2 in the region $\{x>0\}$ initially, then each species will invade its favourable region.

\begin{remark}
    Let us comment on the condition \eqref{eq:omegaini} on the initial data, which may be seen as a condition on non-smallness of the initial condition. It expresses the fact that in a bistable dynamics the initial data should be large enough on a large enough domain to guarantee the invasion of the favourable species. 
    The question of initiating the invasion for a bistable reaction-diffusion equation in a homogeneous environment has been tackled by several mathematicians, in particular since the seminal paper \cite{Zlatos}, see also \cite{Matano, Muratov} and the question of optimizing the initial data has been also addressed more recently in \cite{Ana,Ana2}.
    Similar conditions involving the \textit{bubbles} have been used to study the invasion of a population of mosquitoes in \cite{BartonTurelli, Bliman, Zubelli}.
\end{remark}

Before proving this result, we prove Lemma \ref{lem:bubbleF} and Lemma \ref{lem:bubbleU}. Since the proof is similar for both, we only give the proof of Lemma \ref{lem:bubbleF}. Then we state some preliminary results useful for the proof of Theorem \ref{BB}.
\\

\textit{Proof of Lemma \ref{lem:bubbleF}.}
Let $\alpha \in (\theta^F,K_1^F(1-\frac{d_1}{b_1}))$, we consider the Cauchy problem, on $(0,+\infty)$,
$$
\chi' = -\sqrt{2(F^F(\alpha)-F^F(\chi))}, \qquad \chi(0) = \alpha.
$$
There exists a unique solution to this Cauchy problem denoted  by $\chi_\alpha^{0,F}$. 
Clearly, we have $(\chi_\alpha^{0,F})'<0$ as long as $\chi_\alpha^{0,F} < \alpha$ and, by definition of $L_\alpha^F$,
$\chi_\alpha^{0,F}(L_\alpha^F) = -\max(K_2^F,K_2^U)(1-\frac{d_2}{b_2})$.
It allows  us to define $\chi_\alpha^{0,F}$ on $(0,L_\alpha^F)$ nonincreasing which verifies 
$$\frac 12 ((\chi_\alpha^{0,F})')^2+F^F(\chi_\alpha^{0,F}) = F^F(\alpha).$$
Deriving this latter equality, we obtain $(\chi_\alpha^{0,F})'' + f^F(\chi_\alpha^{0,F}) = 0$.
Then we extend $\chi_\alpha^{0,F}$ by symmetry on $(-L_\alpha^F,L_\alpha^F)$, and on $\IR$ by setting 
$$\chi_\alpha^{0,F}(x) = -\max(K_2^F,K_2^U)(1-\frac{d_2}{b_2})\text{ on } \IR \setminus (-L_\alpha^F,L_\alpha^F).$$
\qed

\begin{lemma}\label{lem:omega_monotone_F}
    Assume that \eqref{hyp:bd} holds and $\gamma^F>0$. Let $\widetilde{\omega}$ be a bounded solution of $$-\widetilde{\omega}'' = f(x,\widetilde{\omega}).$$ 
    If there exists $\widetilde{y}<0$ such that $\theta_F < \widetilde{\omega}(\widetilde{y}) $, where $\theta_F$ is defined in Lemma \ref{lem:bubbleF}. 
    Then, 
    $$
    \lim_{x\to -\infty} \widetilde{\omega}(x) = K_1^F\left(1-\frac{d_1}{b_1}\right).
    $$
\end{lemma}
\begin{proof}
    Let us assume that there exists $\widetilde{y}<0$ such that $\widetilde{\omega}(\widetilde{y}) >\theta_F$ and $\widetilde{\omega}(\widetilde{y}) < K_1^F(1-\frac{d_1}{b_1})$.
    Then, we have $\widetilde{\omega}<K_1^F(1-\frac{d_1}{b_1})$ on $(-\infty,\widetilde{y})$. Indeed, if there exists $z<\widetilde{y}$ such that $\widetilde{\omega}(z)=K_1^F(1-\frac{d_1}{b_1})$ and $\widetilde{\omega}(x)<K_1^F(1-\frac{d_1}{b_1})$ for all $x\in (z,\widetilde{y}]$. Then $\widetilde{\omega}'(z)\leq 0$, but if $\widetilde{\omega}'(z)= 0$, then by uniqueness of the solution of the Cauchy problem satisfied by $\widetilde{\omega}$, we must have $\widetilde{\omega}=K_1^F(1-\frac{d_1}{b_1})$ which is a contradiction.
    Hence $\widetilde{\omega}'(z)< 0$, however when $\omega \geq K_1^F(1-\frac{d_1}{b_1})$, we have $f^F(\omega)\leq 0$, then the solution to $-\widetilde{\omega}'' = f^F(\widetilde{\omega})$ is convex with $\widetilde{\omega}'(z)<0$  This implies that $$\displaystyle lim_{x\to -\infty} \widetilde{\omega}(x) = +\infty,$$ which is a contradiction with the fact that $\widetilde{\omega}$ is bounded.
    
    Then, let us prove that $\widetilde{\omega}$ is nonincreasing on $(-\infty,\widetilde{y})$. By contradiction, if $\widetilde{\omega}$ is not nonincreasing on $(-\infty,\widetilde{y})$, then there exists $z\in (-\infty,\widetilde{y})$ such that $\widetilde{\omega}(z)>\theta_F$ and $\widetilde{\omega}'(z)>0$.
    Thus it is a solution of the Cauchy problem on $(-\infty,z)$
    $$
    -\widetilde{\omega}'' = f^F(\widetilde{\omega}), \qquad \widetilde{\omega}(z) >\theta_F, \quad \widetilde{\omega}'(z) >0.
    $$
    This solution satisfis, for all $x\in(-\infty,z)$,
    $$
    \frac 12 (\widetilde{\omega}'(x))^2 + F^F(\widetilde{\omega}(x)) = \frac 12 (\widetilde{\omega}'(z))^2 + F^F(\widetilde{\omega}(z)).
    $$
    Since $\theta_F<\widetilde{\omega}(z)<K_1^F(1-\frac{d_1}{b_1})$, then for any $\omega<\widetilde{\omega}(z)$, we have
    $F^F(\omega)<F^F(\widetilde{\omega}(z))$ (see Figure \ref{fig:FF}). We deduce from the latter equality that for any $x<z$, we have
    $$
    (\widetilde{\omega}'(x))^2 > (\widetilde{\omega}'(z))^2 > 0.
    $$
    We deduce that $\widetilde{\omega}$ is strictly increasing on $(-\infty,z)$ with a uniformly positive lower bound on the derivative. This is in contradiction with the fact that $\widetilde{\omega}$ is bounded.

    Hence,  $\widetilde{\omega}$ is nonincreasing, and since it is also bounded, it admits a limit at $-\infty$. From the equation satisfied by $\widetilde{\omega}$, this limit should be a root of $f^F$ and should be greater than $\theta_F$. The only possibility is $K_1^F(1-\frac{d_1}{b_1})$. It concludes when $\widetilde{\omega}(\widetilde{y})<K_1^F(1-\frac{d_1}{b_1})$.

    Now, if for all $y<0$, we have $\widetilde{\omega}(y) \geq K_1^F(1-\frac{d_1}{b_1})$, then by definition of $f^F$, we have $$  f^F(\widetilde{\omega}) = -\widetilde{\omega}''  \leq 0 \text{ on } (-\infty,0). $$
    Hence $\widetilde{\omega}$ is convex and bounded on $(-\infty,0)$. Thus $\widetilde{\omega}$ is nondecreasing and converges as $x$ goes to $-\infty$ to the unique root of $f^F$ greater than $K^F_1(1-\frac{d_1}{b_1})$.
\end{proof}

The following lemma may be obtained similarly~:
\begin{lemma}\label{lem:omega_monotone_U}
    Assume that \eqref{hyp:bd} holds and $\gamma^U<0$. Let $\widetilde{\omega}$ be a bounded solution of $$-\widetilde{\omega}'' = f(x,\widetilde{\omega}).$$
    If there exists $\widetilde{y}>0$ such that $\widetilde{\omega}(y) < \theta_U$, where $\theta_U$ is defined in Lemma \ref{lem:bubbleU}. 
    Then, 
    $$
    \lim_{x\to +\infty} \widetilde{\omega}(x) = -K_2^U\left(1-\frac{d_2}{b_2}\right).
    $$
\end{lemma}

As a consequence of these two Lemmas, we have the following existence and uniqueness result~:
\begin{proposition}
    \label{prop:stat}
    Assume that \eqref{hyp:bd} holds and that $\gamma^F>0$ and $\gamma^U<0$. 
    There exists a unique bounded stationary solution $\baromega$ of the problem 
    $$
    -\baromega'' = f(x,\baromega), \qquad \text{on } \IR,
    $$
    such that there exist $y_F<0$ and $y_U>0$ such that $\theta_F<\baromega(y_F)$ and $\baromega(y_U)<\theta_U$.

    Moreover, $\baromega$ is decreasing on $\IR$.
\end{proposition}
\begin{proof}
From Lemma \ref{lem:omega_monotone_F} and \ref{lem:omega_monotone_U}, if such a solution exists, then necessarily, we have 
$$
\lim_{x\to -\infty} \baromega(x) = K_1^F\left(1-\frac{d_1}{b_1}\right), \quad \text{ and }\  
\lim_{x\to +\infty} \baromega(x) = -K_2^U\left(1-\frac{d_2}{b_2}\right).
$$
That is, it is a solution to \eqref{BB-ssamia}.
Then, we split the domain into $\{x<0\}$ and $\{x>0\}$. 

On the set  $\{x>0\}$, the problem reads
$$ 
- \baromega'' = f^U(\baromega), \qquad 
  \displaystyle\baromega(+ \infty)  = -K_2^{U}\big(1 - \frac{d_2}{b_2}\big).
$$
Then, the energy $\mathbf{{E}} = \frac{1}{2} (\baromega')^2 + F^U(\baromega)$ is constant.
In particular, $\mathbf{{E}}(x) = \mathbf{{E}}(+ \infty)$, from which we deduce  that, 
$$
(\baromega')^2 =   2 F^U\left(-K_2^U\left(1-\frac{d_2}{b_2}\right)\right) - 2 F^U(\baromega).
$$
Since $F^U$ is maximal at the point $-K_2^U\left(1-\frac{d_2}{b_2}\right)$, we deduce that $\baromega'$ cannot vanish on $(0,+\infty)$ except for $\baromega = -K_2^U(1-\frac{d_2}{b_2})$.
However, if there exists $\zeta<0$ such that $\baromega(\zeta)=-K_2^U(1-\frac{d_2}{b_2})$ and $\baromega'(\zeta)=0$, then by uniqueness of the Cauchy problem, we have $\baromega(x) = -K_2^U(1-\frac{d_2}{b_2})$ for all $x\in [0,+\infty)$.

By the same  argument  , on the set $\{x<0\}$, we have
$$
(\baromega')^2 =   2 F^F\left(K_1^F\left(1-\frac{d_1}{b_1}\right)\right) - 2 F^F(\baromega).
$$
Since $F^F$ is maximal at the point $K_1^F\left(1-\frac{d_1}{b_1}\right)$, we deduce as above that $\baromega'$ cannot vanish on $(0,+\infty)$ except if $\baromega = K_1^F(1-\frac{d_1}{b_1})$ on $(-\infty,0]$.

However, from the above observation, if $\omega'(0) = 0$, we should have $\omega(0)=-K_2^U(1-\frac{d_2}{b_2})$ and $\omega(0) = K_1^F(1-\frac{d_1}{b_1})$ which is clearly not possible. 
Hence, $\baromega'$ cannot vanish, thus is monotonous and therefore decreasing, given the limits to infinity. 
We have proved that any bounded stationary solution such as in the statement of the Proposition \ref{prop:stat} is decreasing on $\IR$.

 Now, let us show the existence of such a solution. 
Let us introduce for some $\underline{\omega} > - K_2^U (1-\frac{d_2}{b_2}),$ the Cauchy problem on $(0,+\infty)$ 
\begin{align}\label{32}
 \left\lbrace
\begin{array}{lcl}
 \displaystyle \baromega' =  -\sqrt{ 2 F^U\left(-K_2^U\left(1-\frac{d_2}{b_2}\right)\right) -2 F^U(\baromega) }\ , \\[4mm]
 \baromega(0)  = \underline{\omega}.
 \end{array}
\right. 
\end{align}
The solution of this Cauchy problem is well-defined on $(0,+\infty)$ since $F^U$ is maximal at $-K_2^U(1-\frac{d_2}{b_2})$.
Moreover, since the constant $-K_2^U(1-\frac{d_2}{b_2})$ is a solution of this Cauchy problem, by uniqueness, we have $\baromega>-K_2^U(1-\frac{d_2}{b_2})$. Finally, when $x \rightarrow + \infty $, $\baromega$ should converge (since it is nonincreasing and bounded from below) towards a zero of the right hand side, which is $-K_2^U(1-\frac{d_2}{b_2})$.

On the set $\{x<0\}$, by the same token, let us introduce, for some $\underline{\omega} <  K_1^F (1-\frac{d_1}{b_1})$, the Cauchy problem on $(-\infty,0)$ 
\begin{align}\label{3002}
 \left\lbrace
\begin{array}{lcl}
 \displaystyle \baromega' =  -\sqrt{ 2 F^F\left(K_1^F\left(1-\frac{d_1}{b_1}\right)\right) -2 F^F(\baromega) }, \\[4mm]
 \baromega(0)  = \underline{\omega}.
 \end{array}
\right. 
\end{align}
As above, since $F^F$ is maximal at $K_1^F\left(1-\frac{d_1}{b_1}\right)$, we deduce that $\baromega$ is well-defined on $(-\infty,0)$, is nonincreasing and goes to $K_1^F\left(1-\frac{d_1}{b_1}\right)$ as $x$ goes to $-\infty$.

 Finally, we have constructed a solution on $(-\infty,0)$ and on $(0,+\infty)$. We are left to verify that we may find $\underline{\omega}\in (-K_2^U(1-\frac{d_2}{b_2}),K_1^F(1-\frac{d_1}{b_1}))$ such that this solution is differentiable at $ x = 0$, i.e. $ \baromega'(0^+) = \baromega'(0^-)$.
Consider the  function $H$ defined  by
 $$ H(\omega)= F(\omega) - C, $$
 where 
$$ 
F(\omega) = F^F(\omega)-F^U(\omega), \quad
and \quad
 C =  F^F(K_1^F(1-\frac{d_1}{b_1}))-F^U(-K_2^U(1-\frac{d_2}{b_2})). 
$$
Clearly $H$ is continuous and differentiable. Moreover, we compute
\begin{align*}
H(K_1^F(1-\frac{d_1}{b_1})) & =  F^F(K_1^F\big(1-\frac{d_1}{b_1})) -  F^U(K_1^F\big(1-\frac{d_1}{b_1}))  \\
 \quad & -  F^F(K_1^F(1-\frac{d_1}{b_1}))+F^U(-K_2^U(1-\frac{d_2}{b_2})) \\ 
& = F^U(K_1^U\big(1-\frac{d_1}{b_1}))-F^U(K_1^F\big(1-\frac{d_1}{b_1}))-\gamma^U. 
\end{align*}
We recall (see Figure \ref{fig:FU}) that $F^U$ has two local maxima in $-K_2^U(1-\frac{d_2}{b_2})$ and in $K_1^U(1-\frac{d_1}{b_1})$. Hence $F^U(K_1^U\big(1-\frac{d_1}{b_1}))>F^U(K_1^F\big(1-\frac{d_1}{b_1}))$.
Using also assumption \eqref{BBbb-ssamia}, we deduce that $H(K_1^F(1-\frac{d_1}{b_1}))>0$.
By the same token, we have
\begin{align*}
H(-K_2^U(1-\frac{d_2}{b_2})) & =  F^F(-K_2^U(1-\frac{d_2}{b_2})) -  F^U(-K_2^U(1-\frac{d_2}{b_2}))  \\
& \quad -  F^F(K_1^F(1-\frac{d_1}{b_1}))+F^U(-K_2^U(1-\frac{d_2}{b_2})) \\ 
&= F^F(-K_2^U(1-\frac{d_2}{b_2})) - F^F(-K_2^F(1-\frac{d_2}{b_2})) - \gamma^F <0.
\end{align*}
We deduce that  the function $H$ changes sign at least once on the interval $(-K_2^U(1-\frac{d_2}{b_2}),K_1^F(1-\frac{d_1}{b_1}))$.
 So there exists $-K_2^U(1-\frac{d_2}{b_2})<\underline{\omega}<K_1^F(1-\frac{d_1}{b_1})$  such that $H(\underline{\omega})=0$, which implies that $$ F^F(K_1^F(1-\frac{d_1}{b_1})) - F^F(\underline{\omega}) = F^U(-K_2^U(1-\frac{d_2}{b_2})) - F^U(\underline{\omega}). $$
From \eqref{32} and \eqref{3002}, it implies that $\baromega'(0^+) = \baromega'(0^-)$.
 This concludes the proof of existence.

To prove the uniqueness, we have seen above that every stationary solution such as in the statement of the Proposition \ref{prop:stat} verifies \eqref{32} on $(-\infty,0)$ and \eqref{3002} on $(0,+\infty)$. 
Hence, we are left to show the uniqueness of the root of the function $H$.
After straightforward computations, yield
$$
H'(\omega) = f^F(\omega)-f^U(\omega) = b_1\left(\dfrac{1}{K_1^U}-\dfrac{1}{K_1^F}\right) \omega_+^2 + b_2 \left(\dfrac{1}{K_2^U}-\dfrac{1}{K_2^F}\right) \omega_-^2.
$$
Hence, depending on the sign of $K_1^U-K_1^F$ and $K_2^U-K_2^F$, the function $H$ may be monotonous or may change its monotony only once. Since $H(-K_2^U(1-\frac{d_2}{b_2}))<0<H(K_1^F(1-\frac{d_1}{b_1}))$, it implies that $H$ admits only one root. Then the solution is unique.
\end{proof}

\textit{Proof of Theorem \ref{BB}.}
The existence and uniqueness of the solution of \eqref{BB-ssamia} is a consequence of Proposition \ref{prop:stat}.
We are left to show that for any $x\in\IR$, we have $$\lim_{t\to +\infty} \omega(t,x) = \baromega(x)$$ where $\omega$ is a solution of \eqref{eq:omegat} with initial data $\omega^0$ verifying \eqref{eq:omegaini}.
We first notice that $\chi_\alpha^{0,F}(\cdot+x_\alpha)$, defined in Lemma \ref{lem:bubbleF} is a subsolution of the stationary equation. Indeed, on $(-L_\alpha^F-x_\alpha,L_\alpha^F-x_\alpha)$ it is clear by definition, on $\IR\setminus (-L_\alpha^F-x_\alpha,L_\alpha^F-x_\alpha)$ we have $f\big(x,-\max(K_2^F,K_2^U)(1-\frac{d_2}{b_2})\big)\geq 0$, and $(\chi_\alpha^{0,F})''(\pm L_\alpha^F) \geq 0$ in the sense of distribution.
Thus, if we denote $\chi_\alpha^F$ the solution of \eqref{eq:omegat} with initial data $\chi_\alpha^{0,F}(\cdot+x_\alpha)$, i.e.
$$
\partial_t \chi_\alpha^F - \partial_{xx} \chi_\alpha^F = f(x,\chi_\alpha^F), \qquad \chi_\alpha^F(0,x) = \chi_\alpha^{0,F}(x+x_\alpha),
$$
then for any $x\in\IR$, we have $t\mapsto \chi_\alpha^F(t,x)$ is nondecreasing. 
Indeed, we have  that $\partial_t \chi_\alpha^F(0,x) \geq 0$ and by the maximum principle, for all $t\geq 0$, $\partial_t \chi_\alpha^F(t,x) \geq 0$.

By the same manner, let us define $\chi_\beta^U$ the solution of
$$
\partial_t \chi_\beta^U - \partial_{xx} \chi_\beta^U = f(x,\chi_\beta^U), \qquad \chi_\beta^U(0,x) = \chi_\beta^{0,U}(x-y_\beta).
$$
Then, for any $x\in \IR$, $t\mapsto \chi_\beta^U(t,x)$ is nonincreasing.

Moreover, by assumption we have for any $x\in\IR$
$$
\chi_\alpha^{0,F}(x+x_\alpha) \leq \omega^0 \leq \chi_\beta^{0,U}(x-y_\beta).
$$
Hence by comparison principle, for any $t>0$ and any $x\in \IR$,
\begin{equation}
    \label{bornomega}
    \chi_\alpha^{F}(t,x) \leq \omega(t,x) \leq \chi_\beta^{U}(t,x).
\end{equation}

Finally, $t\mapsto \chi_\alpha^F(t,x)$ is nondecreasing and bounded, then it converges as $t$ goes to $+\infty$.
Necessarily the limit is a stationary solution. From Proposition \ref{prop:stat}, we deduce by uniqueness that 
$\lim_{t\to +\infty} \chi_\alpha^F(t,x) = \baromega(x)$ for any $x\in \IR$. 
By the same token, $t\mapsto \chi_\beta^U(t,x)$ is nonincreasing and bounded, then converges to $\baromega(x)$ for any $x\in\IR$. 
Finally, the proof is complete by using \eqref{bornomega}.
\qed

\subsection{Large time asymptotics when only one species is invasive}

Theorem \ref{BB} investigates the case where the two species are invasive in two different area, which correspond to the case $\gamma^F>0$ and $\gamma^U<0$ (these quantities being defined in \eqref{BBbb-ssamia}). 
In this part, we consider the case where one species is invasive everywhere. To fix the notation we consider that species 1 is invasive everywhere, i.e. we assume
\begin{equation}
    \label{assumption2}
    \gamma^F >0, \qquad \text{and } \quad \gamma^U>0.
\end{equation}
Following the strategy developed above, we may study the large time behaviour of the solution of \eqref{eq:omegat} in this setting. Since we have $\gamma^F>0$, the result of Lemma \ref{lem:bubbleF} is still available. However, Lemma \ref{lem:bubbleU} should be modified as below.
\begin{lemma}
    \label{lem:bubbleU2}
    Assume \ref{hyp:bd} holds. Let us assume that $\gamma^U>0$ and denote $\theta_1^U$ the unique real in $(0,K_1^U(1-\frac{d_1}{b_1})$ such that $F^U(\theta_1^U) = F^U\left(-K_2^U\left(1-\frac{d_2}{b_2}\right)\right)$. 
    Then for any $\alpha\in (\theta_1^U,K_1^U(1-\frac{d_1}{b_1}))$, there exists a family of bubbles, denoted $\chi_{1,\alpha}^{0,U}$, such that $\chi_{1,\alpha}^{0,U}$ is even, nonincreasing on $(0,+\infty)$ and
    $$
    \begin{cases}
        &-(\chi_{1,\alpha}^{0,U})'' = f^U(\chi_{1,\alpha}^{0,U}) \qquad \text{ on } (-L_{1,\alpha}^U,L_{1,\alpha}^U), \\[3mm]
        &\chi_{1,\alpha}^{0,U}(0) = \alpha, \quad \chi_{1,\alpha}^{0,U} (\pm L_{1,\alpha}^U) = -\max(K_2^U,K_2^F)\left(1-\dfrac{d_2}{b_2}\right), \\[3mm]
        &\chi_{1,\alpha}^{0,U}(x) = -\max(K_2^F,K_2^U)\left(1-\dfrac{d_2}{b_2}\right) \quad \text{ on } \IR \setminus (-L_{1,\alpha}^U,L_{1,\alpha}^U),
    \end{cases}
    $$
    where
    $$
    L_{1,\alpha}^U := \int_{-\max(K_2^F,K_2^U)(1-\frac{d_2}{b_2})}^\alpha \frac{dz}{\sqrt{2(F^U(\alpha)-F^U(z))}}. 
    $$
\end{lemma}
In this case, the main result reads as follows.
\begin{theorem}
    \label{TH2}
With the above notations, let us assume that \eqref{hyp:bd} and \eqref{assumption2} hold.
Let us assume that the initial data $\omega^0$ is such that, $\omega^0\in L^\infty(\IR)$ and there exist $\alpha\in (\theta^F,K_1^F(1-\frac{d_1}{b_1}))$, $x_\alpha > L_\alpha^F$, and $\alpha_1\in (\theta_1^U,K_1^U(1-\frac{d_1}{b_1}))$, $y_{\alpha_1}>L_{1,\alpha_1}^U$ such that, for all $x\in\IR$,
\begin{equation}
    \label{eq:omegaini1}
    \chi_\alpha^{0,F}(x+x_\alpha) + \chi_{1,\alpha_1}^{0,U}(x-y_{\alpha_1}) \leq \omega^0(x).
\end{equation}
Then, the solution of the equation \eqref{eq:omegat} converges as $t$ goes to $+\infty$, for any $x\in\IR$, towards the unique stationary solution $\baromega$ of the problem
\begin{align}\label{BB-ssamia1}
 \left\lbrace
\begin{array}{lcl}
 - \baromega'' = f(x,\baromega), \quad \text{ on } \IR, \\[2mm]
  \displaystyle\baromega(-\infty) = K_1^{F}\left(1 - \frac{d_1}{b_1}\right), \qquad 
  \displaystyle\baromega(+ \infty)  = K_1^{U}\left(1 - \frac{d_1}{b_1}\right).
 \end{array}
\right. 
\end{align}
Moreover, $\baromega$ is monotonous on $\mathbb{R}$.

\end{theorem}

\begin{proof}
    The proof of this result follows the one of Theorem \ref{BB}. Therefore, we do not detail it but gives only the main changes. 

    Let us consider $\baromega$ a bounded stationary solution of 
    $$
    -\baromega'' = f(x,\baromega), \qquad \text{ on } \IR,
    $$
    such that there exists $y_F<0$ and $y_U>0$ such that $\theta_F<\baromega(y_F)$ and $\theta_U<\baromega(y_U)$. 
    If such a solution exists, since $\gamma^F>0$ and $\gamma^U>0$, we may apply Lemma \ref{lem:omega_monotone_F} to find that 
    $$
    \lim_{x\to -\infty} \baromega(x) = K_1^F\left(1-\frac{d_1}{b_1}\right), \qquad \text{and }\quad
    \lim_{x\to +\infty} \baromega(x) = K_1^U\left(1-\frac{d_1}{b_1}\right).
    $$
    Then, following the proof of Proposition \ref{prop:stat}, we may show similarly that there exists a unique stationary solution to \eqref{BB-ssamia1}.
    It shows then the existence and the uniqueness of a solution to the stationary problem such that there exists $y_F<0$ and $y_U>0$ such that $\theta_F<\baromega(y_F)$ and $\theta_U<\baromega(y_U)$. 

    To study the large time convergence, we proceed as above and remark that $\chi_\alpha^{0,F}(\cdot+x_\alpha)+\chi_{1,\alpha_1}^{0,U}(\cdot-y_{\alpha_1})$ is a subsolution of the stationary equation (after noticing that the compact supports of these two functions are disjoint). Then, if we denote $\chi$ the solution of
    $$
    \partial_t \chi - \partial_{xx} \chi = f(x,\chi), \qquad \chi(0,x) = \chi_\alpha^{0,F}(x+x_\alpha)+\chi_{1,\alpha_1}^{0,U}(x-y_{\alpha_1}).
    $$
    Then, for all $t\geq 0$, we have $\partial_t \chi(t,x)\geq 0$.
    For the supersolution, we take $\chi_1$ solution of
    $$
    \partial_t \chi_1 - \partial_{xx} \chi_1 = f(x,\chi_1), \qquad \chi_1(0,x) = \max\{\|\omega_0\|_\infty, K_1^F,K_1^U\}.
    $$
    Then, for any $t\geq 0$, we have $\partial_t \chi_1(t,x)\leq 0$. 

    Moreover, by assumption \eqref{eq:omegaini1}, we have
    $$
    \chi(0,x) \leq \omega^0(x) \leq \chi_1(0,x).
    $$
    Hence, by comparison principle, for any $t>0$ and $x\in\IR$, we have
    $$
    \chi(t,x) \leq \omega(t,x) \leq \chi_1(t,x).
    $$
    We conclude as in the previous section~: The subsolution is nondecreasing and bounded thus converges towards a stationary solution, since it has been proved that such a solution is unique, we deduce that the limit is $\baromega$. Similarly, the supersolution is nonincreasing and bounded thus converges to $\baromega$. 
\end{proof}

\subsection{Numerical illustrations}\label{sec:num}
In this subsection  we provide some numerical illustrations of our theoretical results. Then we discretize system \eqref{syst12} thanks to a semi-implicit finite difference method.
The domain is assumed to be given by $[-40,40]$ and the choice of the numerical parameters are freely adapted from \cite{Strugarek} and are given in Table \ref{tablenum} below.
Since we have to take a bounded domain for the numerical simulation, we impose Neumann boundary conditions at the boundary of the domain.

In order to illustrate the two possible scenarios of Theorem \ref{BB} and Theorem \ref{TH2}, we take different values for the carrying capacities~: In a first case, we make the following choice
$$
K_1^F = 10, \quad K_1^U = 1, \quad K_2^F = 1, \quad K_2^U = 10.
$$
In this situation, we have $\gamma^F>0$ and $\gamma^U<0$. The initial data are given by
$$
n_1^0 (x) = 2. \,\mathbf{1}_{[10,20]}, \qquad
n_2^0 (x) = 2.5 \,\mathbf{1}_{[-20,-10]}.
$$
The time dynamics is displayed in Figure \ref{fig:test1}. We observe that, as stated in Theorem \ref{BB}, species $n_1$ invades the region $\{x>0\}$, whereas $n_2$ invades the region $\{x<0\}$.
At final time each species is installed in a region and is segregated with the others species.
The density at final time is plotted in Figures \ref{fig:stat1} and \ref{fig:stat2}, which then represent the unique stationary solution as stated in Theorem \ref{BB}.

\begin{figure}[h!]
\centering\includegraphics[width=0.49\linewidth,height=5.7cm]{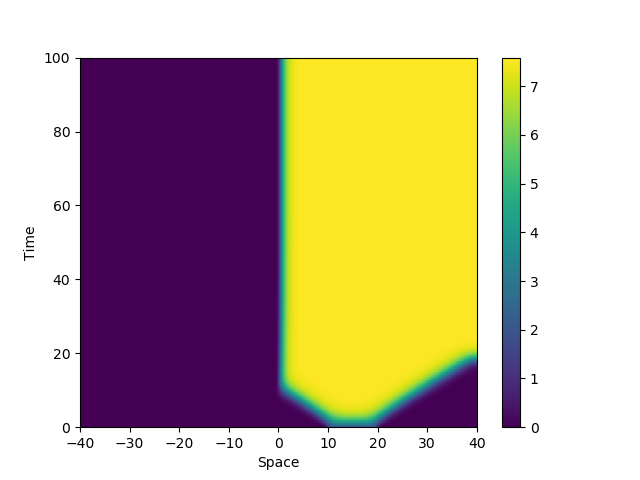} 
\includegraphics[width=0.49\linewidth,height=5.7cm]{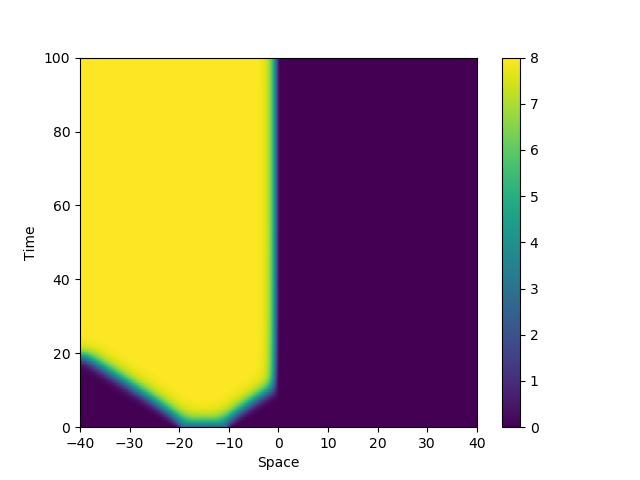}   
\caption{Time dynamics of functions $n_1$ (left) and $n_2$ (right) when conditions of Theorem \ref{BB} are satisfied, in particular $\gamma^F>0$ and $\gamma^U<0$ (see \eqref{BBbb-ssamia}. In this case, species $n_1$ and species $n_2$ are installed in two different regions of the domain.}
\label{fig:test1}
\end{figure}

In a second case, we make the following choice
$$
K_1^F = 10, \quad K_1^U = 7, \quad K_2^F = 1, \quad K_2^U = 5.
$$
In this situation, we have $\gamma^F>0$ and $\gamma^U>0$ then species $n_1$ is invasive in both region $\{x>0\}$ and $\{x<0\}$. We take the same initial data as above.
The numerical results are displayed in Figure \ref{fig:test2}. We observe that the species $n_1$ invades the region $\{x<0\}$ by pushing the species $n_2$. This is in line with the result stated in Theorem \ref{TH2}.

\begin{figure}[h!]
\begin{minipage}[C]{.46\linewidth}
\begin{center}
\includegraphics[scale=0.5]{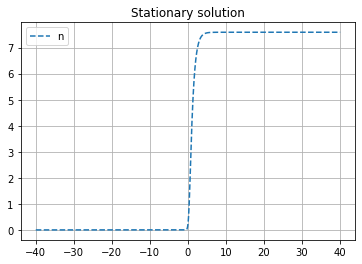}   \end{center}
\caption{Graph of the stationnary solution $n_1$ }
\label{fig:stat1}
\end{minipage} \hfill
\begin{minipage}[C]{.46\linewidth}
\begin{center}
 \includegraphics[scale=0.5]{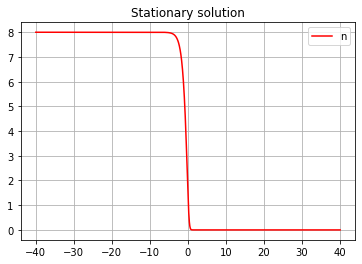}   
\end{center}
\caption{Graph of the stationnary solution $n_2$}
\label{fig:stat2}
\end{minipage} 
\end{figure}

\begin{figure}[h!]
\begin{center}
\includegraphics[width=0.49\linewidth,height=5.7cm]{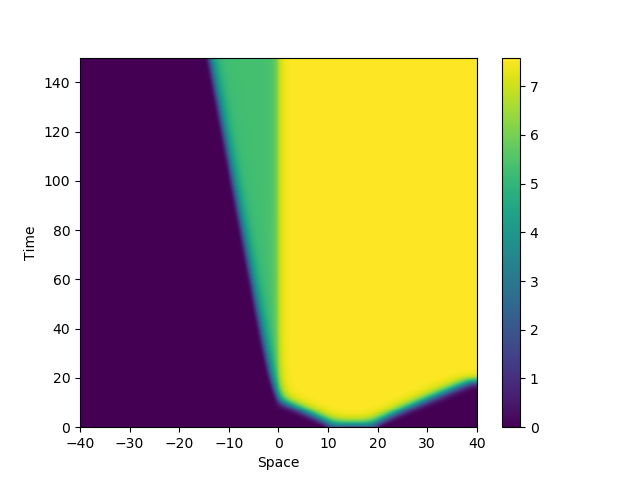}
\includegraphics[width=0.49\linewidth,height=5.7cm]{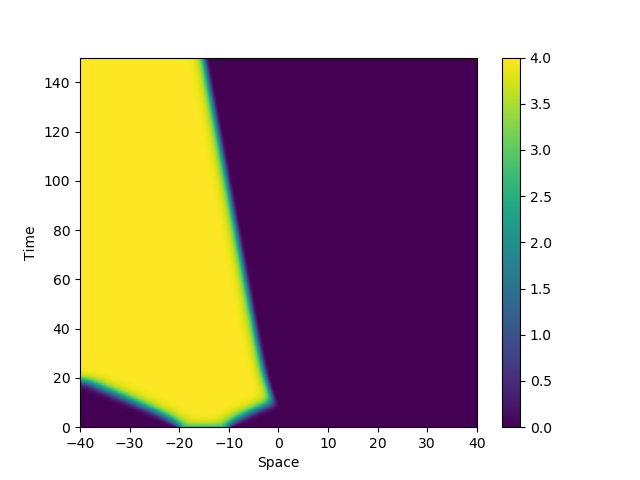}
\caption{Time dynamics of functions $n_1$ (left) and $n_2$ (right) when conditions of Theorem \ref{TH2} are satisfied, in particular $\gamma^F>0$ and $\gamma^U>0$ (see \eqref{BBbb-ssamia}. After invading the region $\{x>0\}$, the species $n_1$ invades the region $\{x<0\}$.}
\end{center}
\label{fig:test2}
\end{figure}

\section{Introducing Wolbachia}\label{sec:Wol}
\subsection{Mathematical model}

In this section, we consider the second model proposed in the introduction \eqref{B-samia}. 
We will use the notations
$$
\mathbf{n}=\left(n_{1}, n_{2} ,n_{3}\right), \quad \text { and } \quad \mathbf{f}=\left(f_{1}, f_{2}, f_{3}\right).
$$
To denote, respectively, the density mosquitoes vector and the right-hand-side functions vector in \eqref{B-samia}.
We first consider the homogeneous situation for which we may assume that $n_1$, $n_2$ and $n_3$ does not depend on the space variable. Then system \eqref{B-samia} simplifies into the following ODE system,
\begin{equation}\label{ACA}
\left\{
   \begin{array}{l}
       \displaystyle n_1' = b_1 n_1\left(1-\frac{n_1}{K_1}\right) - c n_1\, (n_2 + n_3) - d_1 n_1, \\
       \displaystyle n_2' = b_2 n_2 \frac{n_2}{n_2 + n_3}\left(1-\frac{n_2 + n_3}{K_2}\right) - c n_1\, n_2 - d_2 n_2,
       \\
       \displaystyle n_3' = b_3 n_3\left(1-\frac{n_2 + n_3}{K_2}\right) - c n_1\, n_3 - d_3 n_3.
   \end{array}
\right.
\end{equation}
Complemented with nonnegative initial data $n_1(t=0) = n_{1}^{0}$, $n_2(t=0) = n_{2}^{0}$, $n_3(t=0) = n_{3}^{0}$.
Notice that this system in the particular case $n_1=0$ has been studied in \cite{SIMA}.
By a direct application of the Cauchy-Lipschitz theorem, there exists an unique global solution
of this Cauchy problem. Moreover, this solution is non-negative. The Jacobian
associated to the right-hand side of the  above system reads: 
\begin{align*}
& \operatorname{Jac(n_1, n_2, n_3})=  \\
& \begin{pmatrix}
b_1 \left(1-\dfrac{2 n_1}{K_1}\right)- c (n_2 + n_3)-d_1  & -c n_1 & -c n_1 \\[2mm] 
- c n_2 &  b_2 - \dfrac{b_2 n_3^2}{(n_2+n_3)^2}-\dfrac{2 b_2 n_2^2}{K_2(n_2+n_3)}- c n_1-d_2 &  \dfrac{-b_2 n_2^2}{(n_2+n_3)^2} \\[4mm]
- cn_3 & \dfrac{-b_3 n_3}{K_2} & b_3\left(1 - \dfrac{n_2 + 2 n_3}{K_2}\right) - cn_1 - d_3 
\end{pmatrix}.
\end{align*}
Then, the following Lemma provides the equilibria and their stability 
\begin{lemma}\label{14} 
 Let us assume that 
\begin{equation}\label{hypbidi}
b_1>d_1,\quad b_2>d_2,\quad b_3>d_3,\quad b_2 > b_3, \quad d_3 > d_2. 
\end{equation}
Then system $\eqref{ACA}$ admits the  following non-negative steady states:
$$
(0,0,0),\left(n_1^*, 0,0 \right),\left(0, n_2^*,0\right), \left( 0,0, n_3^*, \right) \quad \text { with } \quad n_1^*=K_1\left(1-\frac{d_1}{b_1}\right),\quad n_i^*=K_2\left(1-\frac{d_i}{b_i}\right), i \in\{2,3\} .
$$
In this case, $(0, 0, 0)$
is  (locally linearly) unstable.

Moreover, if
\begin{align}\label{1002}
c>\max \left\{\frac{b_1}{K_1}\left(\frac{b_2-d_2}{b_1-d_1}\right), \frac{b_2}{K_2}\left(\frac{b_1-d_1}{b_2-d_2}\right), \frac{b_3}{K_2}\left(\frac{b_2-d_2}{b_3-d_3}\right)\right\}.
\end{align}
In this case, the two steady states  $\left(n_1^*, 0,0 \right)$ , $\left(0, n_2^*,0\right)$ and $ \left( 0,0, n_3^*, \right)$ are (locally asymptotically ) stable.

Moreover, there exist other distinct non-negative states $(\Bar{n_{21}}, \Bar{n_{12}},0) , (\Bar{n_{31}}, 0 , \Bar{n_{13}}),(0 , \Bar{n_{32}}, \Bar{n_{23}}), (\Bar{n_{41}},\Bar{n_{22}} , \Bar{n_{42}}- \Bar{n_{22}} )$,  where 
\begin{align}\label{23}
\Bar{n_{21}} &= \left( \frac{c(b_2-d_2)- \frac{b_2}{K_2}(b_1-d_1)}{c^2 - \frac{b_1 b_2}{K_1 K_2}}\right)     \quad ,  \Bar{n_{12}}  = \left( \frac{c(b_1-d_1)- \frac{b_1}{K_1}(b_2-d_2)}{c^2 - \frac{b_1 b_2}{K_1 K_2}}\right)  \\ 
 \Bar{n_{31}} &= \left( \frac{c(b_3-d_3)- \frac{b_3}{K_2}(b_1-d_1)}{c^2 - \frac{b_1 b_3}{K_1 K_2}}\right)    \quad,  \Bar{n_{13}}  = \left( \frac{c(b_1-d_1)- \frac{b_1}{K_1}(b_3-d_3)}{c^2 - \frac{b_1 b_3}{K_1 K_2}}\right)  \nonumber \\ 
\Bar{n_{32}} &= \left(\frac{d_2}{b_2} \frac{K_2(b_3-d_3)}{d_3 } \right)   \quad,  \Bar{n_{23}} = \left( K_2 \left(1-\frac{d_3}{b_3}\right)- \frac{d_2}{b_2} \frac{K_2(b_3-d_3)}{d_3 } \right)\nonumber  \\   
 \Bar{n_{22}} &= \frac{K_2}{b_2}\times \frac{ \bigg[\frac{b_1}{K_1}(b_3-d_3)-c(b_1-d_1)\bigg]\bigg[\frac{b_1b_3}{K_1 K_2}d_2 + c \frac{b_3}{K_2}(b_1-d_1)-c^2(b_3+d_2-d_3)\bigg]} {\bigg({\frac{b_1 b_3}{K_1 K_2}- c^2\bigg)\bigg(K_2\bigg(\frac{b_1b_3}{K_1K_2}- c^2\bigg)-\frac{b_1}{K_1}(b_3-d_3)+c(b_1-d_1)\bigg)}} \nonumber  \\
 \Bar{n_{41}} &= \left( \frac{ \frac{b_3}{K_2}(b_1-d_1)-c(b_3-d_3)}{ \frac{b_1 b_3}{K_1 K_2}- c^2 }\right) \quad,  \Bar{n_{42}} = \left( \frac{ \frac{b_1}{K_1}(b_3-d_3)-c(b_1-d_1)}{ \frac{b_1 b_3}{K_1 K_2}- c^2 }\right) \nonumber
\end{align}
 whereas 
  $(\Bar{n_{21}}, \Bar{n_{12}},0) , (\Bar{n_{31}}, 0 , \Bar{n_{13}}),(0 , \Bar{n_{32}}, \Bar{n_{23}}), $ and $(\Bar{n_{41}},\Bar{n_{22}} , \Bar{n_{42}}- \Bar{n_{22}} )$
are locally linearly unstable.
\end{lemma}
\begin{proof}
Equilibria may be computed directly by solving  the system,
$$
\left\{\begin{array}{l}
0=b_1 \overline{n_1}\left(1-\dfrac{\overline{n_1}}{K_1}\right)-c \overline{n_1} ( \overline{n_2} + \overline{n_3})-d_1 \overline{n_1}, \\[4mm]
0=b_2 \overline{n_2} \dfrac{\overline{n_2} }{\overline{n_2} + \overline{n_3}}\left(1-\dfrac{\overline{n_2}+ \overline{n_3}}{K_2}\right)-c \overline{n_1}\, \overline{n_2}-d_2 \overline{n_2}, \\[4mm]
0=b_3 \overline{n_3}\left(1-\dfrac{ \overline{n_2} + \overline{n_3}}{K_2}\right)-c \overline{n_1} \,\overline{n_3}-d_3 \overline{n_3}.
\end{array}\right.
$$
Computing the Jacobian at the equilibria we get. At $(0,0,0)$, we compute the directional limit in direction $(h, k, l)$ as
$$ 
\lim _{t \rightarrow 0} \frac{ \mathbf{f}(t h, t k, tl)}{t}=\left(\begin{array}{c} \left(b_1 - d_1 \right)  h \\
\left( b_{2}  \frac{ k}{ k+ l} -d_{2}\right)  k \\
\left( b_{3} -d_{3} \right)  l 
\end{array}\right)
$$
and in particular we find that the direction  $ (0, 1, 0)$ is unstable. 

\ 

For the equilibrium $\left(n_1^*, 0,0\right)$, we compute for $h_1\in \RR$, $h_2>0$, $h_3>0$ the limit 
 $$
 \lim _{t \rightarrow 0}\frac{
 \mathbf{f}( n_1^* + t h_1, t h_2, th_3)- \mathbf{f}(n_1^*,0,0)}{t} =\left(\begin{array}{c} \left( b_1 (1 - \frac{ 2 n_1^*}{K_1}) - d_1 \right)h_1 \\
\left( b_{2}\frac{h_2}{h_2+h_3} - c n_1^* -d_{2}\right)  h_2 \\
\left( b_{3}- c n_1^* -d_{3} \right)  h_3 
\end{array}\right) = D \begin{pmatrix} h_1 \\ h_2 \\ h_3  \end{pmatrix}
  $$
  where 
  \begin{align*}
  D & = \begin{pmatrix}
      b_1 (1 - \frac{ 2 n_1^*}{K_1}) - d_1  & 0 & 0  \\
      0 & b_{2}\frac{h_2}{h_2+h_3} - c n_1^* -d_{2} & 0 \\
      0 & 0 & b_{3}- c n_1^* -d_{3}
  \end{pmatrix}.
  \end{align*}
  With assumption \eqref{1002}, we have  $b_i-d_i-c n_j^* < 0$. Hence, $D$ is a diagonal matrix with negative eigenvalue. It implies that  $(n_1^*, 0, 0)$ is locally asymptotically stable. 
  
We have also
$$
\operatorname{Jac}\mathbf{(0,n_2^*,0)}=
\begin{pmatrix}
b_1 -d_1 - c n_2^*  & 0 & 0 \\[2mm] 
-c n_2^* & b_2 -d_2 -  \dfrac{2b_2}{K_2}n_2^* & -b_2 \\[4mm]
0 & 0 & b_3 \left(1- \dfrac{n_2^*}{K_2}\right) - d_3  
\end{pmatrix}
$$
and 
$$
\operatorname{Jac}\mathbf{(0,0, n_3^*)}=
\begin{pmatrix}
b_1 -d_1 - c n_3^*  & 0 & 0 \\[2mm] 
0 &  -d_2 & 0 \\[4mm]
-c n_3^* & - \dfrac{ b_3 n_3^*}{K_2} & b_3 - d_3 - \dfrac{ 2 b_3 n_3^*}{K_2} 
\end{pmatrix}.
$$
We have  using with assumption \eqref{hypbidi}, for $i,j\in \{1,2,3\}$, $b_i-d_i-\dfrac{2b_i}{K_i} n_i^* = d_i-b_i < 0$ \, and \, $b_3(1-\frac{n_2^*}{K_2}) - d_3 = (b_3d_2-d_3b_2)/b_2 <0$.
Using assumption \eqref{1002}, we also get $$b_i-d_i-c n_j^* = b_i - d_i - c K_j \dfrac{b_j-d_j}{b_j}<0.$$
The linear stability of  $ \mathbf{(0,n_2^*,0)}$, and $\mathbf{(0,0, n_3^*)}$ follows.  
Then, we have
\begin{align*}
&\operatorname{Jac(\Bar{n_{21}}, \Bar{n_{12}},0})= \\
& \begin{pmatrix}
b_1 \left(1-\frac{2 \Bar{n_{21}}}{K_1}\right)- c \Bar{n_{12}} -d_1  & -c \Bar{n_{21}} & -c \Bar{n_{21}} \\[2mm] 
- c \Bar{n_{12}} & b_2 (1-\frac{2  \Bar{n_{12}}}{K_2})- c \Bar{n_{21}}-d_2 & -b_2  \\[4mm]
0 & 0 & b_3\left(1 - \frac{\Bar{n_{12}} }{K_2}\right) - c\Bar{n_{21}} - d_3
\end{pmatrix}
\end{align*}
and 
\begin{align*}
&\operatorname{Jac(\Bar{n_{31}}, 0 , \Bar{n_{13}}})= \\ 
& \begin{pmatrix} 
b_1 \left(1-\frac{2 \Bar{n_{31}}}{K_1}\right)- c   \Bar{n_{13}}-d_1  & -c \Bar{n_{31}} & -c \Bar{n_{31}}\\[2mm] 
0 &  - c \Bar{n_{31}}-d_2 & 0 \\[4mm]
- c\Bar{n_{13}} & \frac{-b_3 \Bar{n_{13}}}{K_2} & b_3\left(1 - \frac{ 2 \Bar{n_{13}}}{K_2}\right) - c\Bar{n_{31}}- d_3 
 \end{pmatrix}.
\end{align*}
 From the proof of  Lemma \ref{11} 
 we obtain that the point $(\Bar{n_{21}}, \Bar{n_{12}},0)$  and $(\Bar{n_{31}}, 0 , \Bar{n_{13}})$ are (locally linearly) unstable. Indeed, $\operatorname{det}(\operatorname{Jac(\Bar{n_{31}}, 0 , \Bar{n_{13}}})) = -(c\Bar{n_{31}}+d_2) \times  \operatorname{det}\left(\operatorname{Jac}(\overline{n_1}, \overline{n_2})\right) >0 ,(\Bar{n_{31}}>0$ and $ \operatorname{det}\left(\operatorname{Jac}(\overline{n_1}, \overline{n_2})\right) <0 $ ( according to condition \eqref{1002}), it implies the unstability of $(\Bar{n_{31}}, 0 , \Bar{n_{13}})).$\\
Finally, we compute
$$
\begin{aligned}
&\operatorname{Jac(0 , \Bar{n_{32}}, \Bar{n_{23}}}) = \\
& \begin{pmatrix}
b_1 - c (\Bar{n}_{32} + \Bar{n}_{23})- d_1  & 0 & 0 \\[2mm] 
- c \Bar{n}_{32} & b_2 \left(1 -\frac{\Bar{n}_{23}}{\Bar{n}_{23}+\Bar{n}_{32}} + \frac{\Bar{n}_{32}\Bar{n}_{23}}{(\Bar{n}_{32}
+\Bar{n}_{23})^2}\right)-\frac{2 b_2 \Bar{n}_{32}^2}{K_2(\Bar{n}_{32}+\Bar{n}_{23})}-d_2 &  \frac{- b_2 \Bar{n}_{32}^2}{(\Bar{n}_{32}+\Bar{n}_{23})^2}
 \\[4mm]
- c \Bar{n}_{23} & \frac{-b_3 \Bar{n}_{23}}{K_2} & b_3\left(1 - \frac{\Bar{n}_{32} + 2 \Bar{n}_{23}}{K_2}\right)  - d_3 
& \end{pmatrix}.
\end{aligned}
$$
A straightforward computation yield 
$$
\begin{aligned}
&\operatorname{det} (\operatorname{Jac(0 , \Bar{n_{32}}, \Bar{n_{23}}}) ) & = (b_1 - c (\Bar{n}_{32} +\Bar{n}_{23})-d_1) ) \times  \operatorname{det} \left|\begin{array}{cc}
a_{22}   & a_{23}  \\ \\
a_{32}  & a_{33} 
\end{array}\right|
\\
\end{aligned}
$$
with $\Bar{n_{23}}$ and $\Bar{n_{32}}$ defined in $\eqref{23}$. Since  $ \Bar{n_{23}} + \Bar{n_{32}} = K_2 \alpha_3 $ with $\alpha_3 = \left(1-\frac{d_3}{b_3}\right)$,  then   with assumption   \eqref{1002} we get 
\begin{align*}
 a_{11} & =(b_1 - c (\Bar{n_{32}}+ \Bar{n_{23}})-d_1) = b_1 - d_1 - c \frac{K_2}{b_3}(b_3 - d_3),\\
 a_{22} & = d_2 \left( 1 - \frac{d_2 }{b_2 (1 -  \alpha_3)^2}\right) , \quad
 a_{23}  = - \frac{d_2^2 }{b_2 (1 -  \alpha_3)^2} , \\
 a_{32} & = - b_3 \alpha_3 \left( 1 -  \frac{d_2}{b_2(1 -  \alpha_3)}\right), \quad 
 a_{33}  = b_3 \alpha_3 \left( -1 +  \frac{d_2}{b_2(1 -  \alpha_3)}\right).
 \end{align*}
Clearly,  we have $ a_{33} = a_{32}$ and $a_{22} = d_2 + a_{23} $, then 
$ \operatorname{det} (\operatorname{Jac(0 , \Bar{n_{32}}, \Bar{n_{23}}}) )  = a_{11} \times  \operatorname{det}(\operatorname{Jac \left(a_{22}  a_{33}  -  a_{23} a_{32}  \right))} = a_{33}  \left(a_{22} -   a_{23} \right) = a_{33} d_2 $, and the term 
\begin{align}
 \Bar{n_{23}} &  = K_2 \alpha_3 - \frac{d_2}{b_2} \frac{K_2(b_3-d_3)}{d_3 } \nonumber \\
 & =  K_2 \alpha_3 \left ( 1 -  \frac{d_2}{b_2(1 -  \alpha_3)}\right) >0 \nonumber\\
 & \Rightarrow \left(\frac{d_2}{b_2(1 -  \alpha_3)}\right) < 1.
\end{align}
Then $  \operatorname{det}(\operatorname{Jac \left(a_{22} \, a_{33} \, -  a_{23} \, a_{32} \,  \right))}<0 $, this implies that $ \operatorname{det} (\operatorname{Jac(0 , \Bar{n_{32}}, \Bar{n_{23}}}) ) >0.  $ This implies  the unstability for $(0 , \Bar{n_{32}}, \Bar{n_{23}})$, and  the proof of Lemma \ref{14} is complete.
\end{proof}
\subsection{Formal reduction in the strong competition setting}

We follow the strategy of the previous section and consider now the case with spatial diffusion, in the strong competition setting. 
Moreover, we are interested in the case of heterogeneous environment and assume as above that $K_1(x) = K_1^F \mathbf{1}_{\{x<0\}}+K_1^U \mathbf{1}_{\{x>0\}}$ and $K_2(x) = K_2^F \mathbf{1}_{\{x<0\}}+K_2^U \mathbf{1}_{\{x>0\}}$.
 Setting $c=\frac{1}{\varepsilon}\to +\infty$, the spatial model of \textit{Wolbachia} replacement technique reads
\begin{subequations}\label{systW}
\begin{equation}\label{systW:eq1}
\partial_t n_1^\varepsilon - \partial_{xx} n_1^\varepsilon = b_1 n_1^\varepsilon\left(1-\dfrac{n_1^\varepsilon}{K_1(x)}\right) - \frac{1}{\varepsilon} n_1^\varepsilon (n_2^\varepsilon+n_3^\varepsilon) - d_1 n_1^\varepsilon,
\end{equation}
\begin{equation}\label{systW:eq2}
\partial_t n_2^\varepsilon - \partial_{xx} n_2^\varepsilon = b_2 n_2^\varepsilon \frac{n_2^\varepsilon}{n_2^\varepsilon+n_3^\varepsilon}\left(1-\dfrac{n_2^\varepsilon+n_3^\varepsilon}{K_2(x)}\right) - \frac{1}{\varepsilon} n_1^\varepsilon n_2^\varepsilon - d_2 n_2^\varepsilon,
\end{equation}
\begin{equation}\label{systW:eq3}
\partial_t n_3^\varepsilon - \partial_{xx} n_3^\varepsilon = b_3 n_3^\varepsilon \left(1-\dfrac{n_2^\varepsilon+n_3^\varepsilon}{K_2(x)}\right) - \frac{1}{\varepsilon} n_1^\varepsilon n_3^\varepsilon - d_3 n_3^\varepsilon.
\end{equation}
\end{subequations}
 Formally, when $\varepsilon\to 0$, we get $n_1^\varepsilon n_2^\epsilon \to 0$ and $n_1^\varepsilon n_3^\epsilon \to 0$. Hence, if we assume that $n_i^\varepsilon$ converges, at least formally, towards $n_i$, $i\in\{1,2,3\}$, then the supports of $n_1$ and $n_2+n_3$ are disjoints. Following the ideas in Section \ref{sec:reduction}, we introduce the new variables 
 $$
 N = n_2 + n_3, \qquad \omega = n_1 - N, \qquad p = \frac{n_3}{N}.
 $$
 Here $p$ is the fraction of infected mosquitoes, which is defined only when $N>0$, then $0\leq p \leq 1$.
 At the formal limit we get $n_1= \omega_+$ and $N=\omega_-$.
 After straightforward computations, we deduce the equation verified by $\omega$
 \begin{equation}
    \label{eq:omega2}
    \partial_t \omega - \partial_{xx} \omega = g(x,p,\omega),
 \end{equation}
 where 
 \begin{equation}
    \label{eq:gpomega}
    g(x,p,\omega) = b_1 \omega_+\left(1-\dfrac{\omega_+}{K_1(x)}\right) - d_1 \omega_+ - \omega_- \left(\big(b_2(1-p)^2+b_3 p\big) \left(1-\frac{\omega_-}{K_2(x)}\right)-d_2(1-p)-d_3 p\right).
 \end{equation}
 This is coupled with the equation on the proportion, on the set $\{\omega<0\}$,
 $$
 \partial_t p - 2\partial_x p \,\partial_x \ln(\omega_-) - \partial_{xx} p = p(1-p) \left( \big(b_3-b_2(1-p) \big) \left(1-\frac{\omega_-}{K_2(x)}\right) + d_2 -d_3\right).
 $$
We observe immediately that $p=0$ and $p=1$ are steady state solutions to this equation. When $p=0$ the population $n_3$ is absent and we have to deal only with the interaction between the populations $n_1$ and $n_2$. When $p=1$, the population $n_2$ is absent and, in this case, the model describes the interaction between populations $n_1$ and $n_3$. 
We will consider these two cases independently.
We first introduce the following notations~:
\begin{equation}
 \label{notation3}
    \left\{
    \begin{array}{ll}
    g_{12}^F (\omega) = g(x,0,\omega) \mathbf{1}_{\{x<0\}}, \quad & g_{12}^U (\omega) = g(x,0,\omega) \mathbf{1}_{\{x>0\}}, \\[3mm]
    g_{13}^F (\omega) = g(x,1,\omega) \mathbf{1}_{\{x<0\}}, \quad & g_{13}^U (\omega) = g(x,1,\omega) 
    \mathbf{1}_{\{x>0\}}, \\[3mm]
    G_{1i}^F(\omega) = \ds{\int_0^\omega g_{1i}^F(\zeta) \,d\zeta}, \quad & G_{1i}^U(\omega) = \ds{\int_0^\omega g_{1i}^U(\zeta) \,d\zeta}, \quad i\in \{2,3\}.
    \end{array}
    \right.
\end{equation}
Notice that with the notations in \eqref{ccc}, we have $g_{12}^F = f^F$ and $g_{12}^U = f^U$.
As in \eqref{BBbb-ssamia}, we introduce also the constants
\begin{equation}
    \label{gamma3}
    \left\{
    \begin{array}{l}
    \gamma_{1i}^F := G_{1i}^F\left(K_1^F\left(1-\frac{d_1}{b_1}\right)\right) - G_{1i}^F\left(-K_2^F\left(1-\frac{d_i}{b_i}\right)\right), \quad i \in \{2,3\},  \\[3mm]
    \gamma_{1i}^U := G_{1i}^U\left(K_1^U\left(1-\frac{d_1}{b_1}\right)\right) - G_{1i}^U\left(-K_2^U\left(1-\frac{d_i}{b_i}\right)\right).
    \end{array}
    \right.
\end{equation}
The following Lemma gathers some obvious properties on these coefficients under the  natural assumption \eqref{hypbidi}~:
\begin{lemma}
    Let us assume that \eqref{hypbidi} holds, in particular $b_2>b_3$ and $d_2<d_3$. Then we have 
    $$
    g_{12}^F(\omega) < g_{13}^F(\omega) \text{ for } \omega \in (-K_2^F,0),
    \quad g_{12}^U(\omega) < g_{13}^U(\omega) \text{ for } \omega \in (-K_2^U,0),
    $$
    and
    $$
    \gamma_{12}^F < \gamma_{13}^F, \quad \gamma_{12}^U < \gamma_{13}^U.
    $$
\end{lemma}

We assume that the species $1$ is always invasive in the region $\{x<0\}$, from Remark \ref{rem:bistability} it means that we have $\gamma_{12}^F>0$ and $\gamma_{13}^F>0$. 
From the results in Theorem \ref{BB} and in Theorem \ref{TH2}, we make the following observations.
\begin{itemize}
    \item If $\gamma_{13}^U>\gamma_{12}^U > 0$, then the species $n_1$ is invasive also in the region $\{x>0\}$. Then, the steady state solution of \eqref{BB-ssamia1} is asymptotically stable, i.e. the species $n_1$ invades the entire domain.
    \item If $0>\gamma_{13}^U>\gamma_{12}^U $, then the species $n_2$ and $n_3$ are invasive in the region $\{x>0\}$. Thus, if we succeed in replacing the species $n_2$ by the species $n_3$ in this region, the following steady states is asymptotically stable
    \begin{align*}
    \left\lbrace
    \begin{array}{lcl}
    - \baromega'' = f(x,\baromega), \quad \text{ on } \IR, \\[2mm]
    \displaystyle\baromega(-\infty) = K_1^{F}\left(1 - \frac{d_1}{b_1}\right), \qquad 
    \displaystyle\baromega(+ \infty)  = K_2^{U}\left(1 - \frac{d_3}{b_3}\right).
    \end{array}
    \right. 
    \end{align*}
    That is the species $n_1$ invades the region $\{x<0\}$ whereas the species $n_3$ invades the region $\{x>0\}$.
    \item If $\gamma_{13}^U>0>\gamma_{12}^U$, this is the interesting case where $n_1$ "wins the competition against" species $n_3$ but "loses against" species $n_2$. Then, if we succeed in replacing the species $n_2$ by $n_3$, the species $n_1$ will be able to invades the entire domain and the steady state solution of \eqref{BB-ssamia1} will be asymptotically stable. 
    If the replacement does not occur, then the stable steady solution of \eqref{BB-ssamia} will be asymptotically stable.
\end{itemize}

\subsection{Numerical illustrations}\label{sec:Wolnum}
 In this part, we illustrate the above results by some numerical simulations.
 We use the same discretization based on a semi-implicit finite-difference scheme as in section \ref{sec:num} but for system \ref{systW}. The spacial domain is $[-40,40]$ and system \ref{systW} is complemented with Neumann boundary conditions. 
 The numerical values of the parameters are given in Table \ref{tablenum}.
\begin{table}[ht!]
\centering\begin{tabular}{ |c|c|c|c|c|c|c|c| } 
\hline
 $b_1$ & $b_2$ & $b_3$ & $d_1$ & $d_2$ & $d_3$ & $c$ & $D$ \\ 
 \hline
 $1.12$ & $1$ & $0.9$ & $0.27$ & $0.2$ & $0.24$ & $10$ & $0.5$ \\ 
 \hline
\end{tabular}    
\caption{Numerical values of the parameters used in the computations.}
\label{tablenum}
\end{table}

\begin{figure}[ht!]
\centering\includegraphics[width=0.7\linewidth,height=5.7cm]{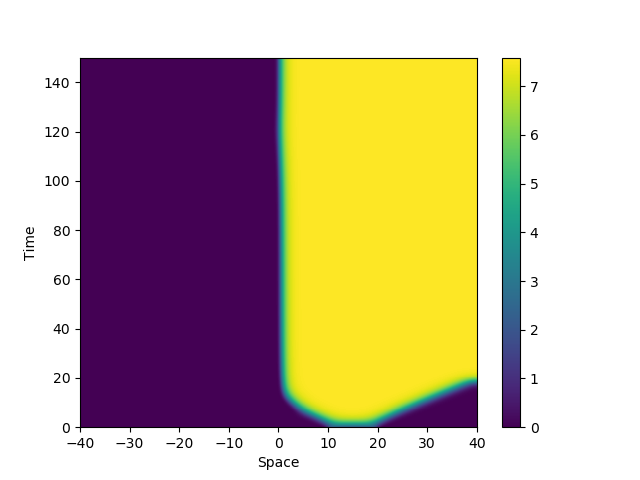}
\includegraphics[width=0.7\linewidth,height=5.7cm]{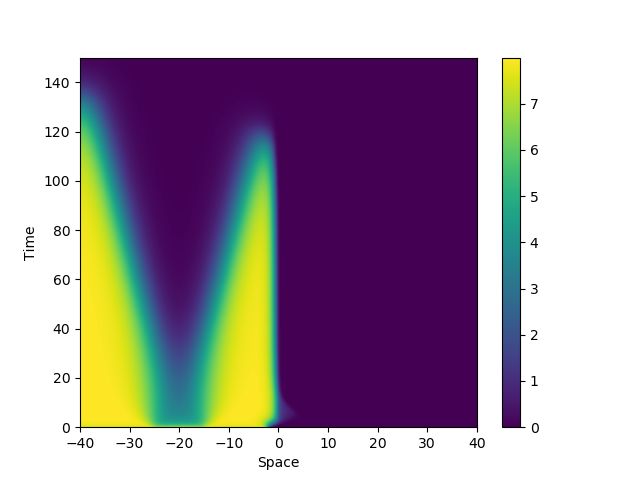}   
\includegraphics[width=0.7\linewidth,height=5.7cm]{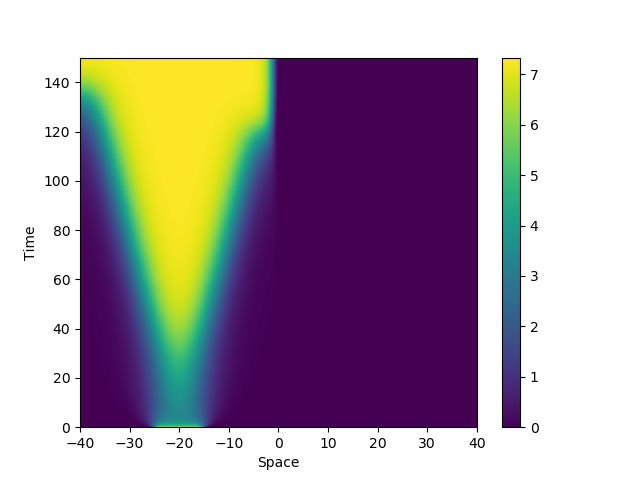}   
\caption{Time dynamics of functions $n_1$ (top), $n_2$ (middle), and $n_3$ (bottom), solution of system \eqref{systW}. In this case the Wolbachia-infected population $n_3$ replaces the wild population $n_2$ in the region $\{x<0\}$, while the population $n_1$ invades the domain $\{x>0\}$.}
\label{fig:Wolb1}
\end{figure}

\begin{figure}[ht!]
\centering\includegraphics[width=0.7\linewidth,height=5.7cm]{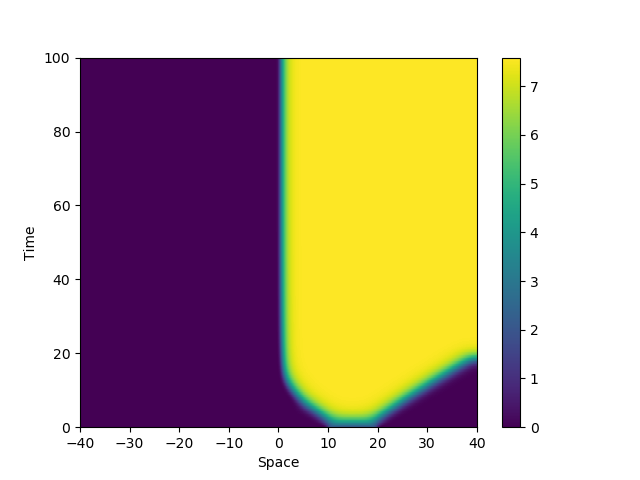}
\includegraphics[width=0.7\linewidth,height=5.7cm]{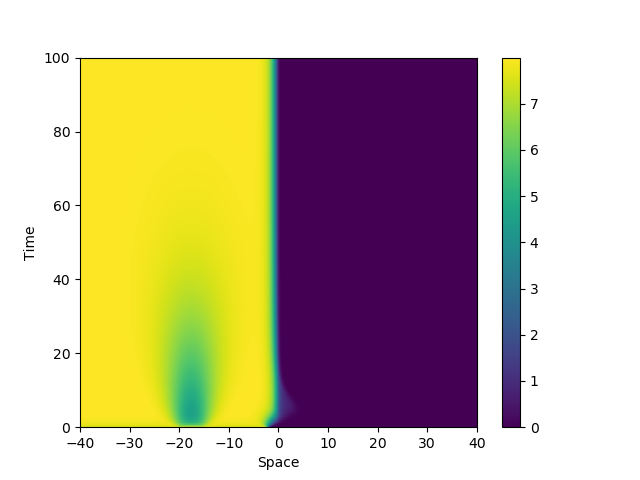}   
\includegraphics[width=0.7\linewidth,height=5.7cm]{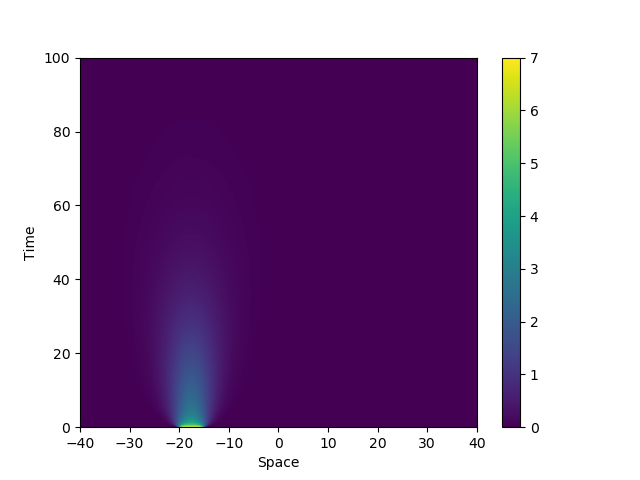}   
\caption{Time dynamics of functions $n_1$ (top), $n_2$ (middle), and $n_3$ (bottom) solution of system \eqref{systW}. In this case the initial condition for the Wolbachia-infected population $n_3$ is not large enough to succeed in replacing the wild population $n_2$ in the region $\{x<0\}$, while the population $n_1$ invades the domain $\{x>0\}$.}
\label{fig:Wolb2}
\end{figure}
 We display in Figures \ref{fig:Wolb1} and \ref{fig:Wolb2} the numerical results obtained for $K_1^F = 10$, $K_1^U = 1$, $K_2^F=1$, and $K_2^F = 10$. For such value, we may verify that $\gamma_{12}^F > 0$ and $0>\gamma_{13}^U>\gamma_{12}^U$. Hence the species $n_1$ is invasive in the region $\{x>0\}$ whereas the species $n_2$ and $n_3$ are invasive in the region $\{x<0\}$. In Figure \ref{fig:Wolb1} the initial data are given by 
 $$
 n_1^0(x) = 2. \,\mathbf{1}_{[10,20]}(x), \qquad
 n_2^0(x) = 7.5 \,\mathbf{1}_{x<0}(x), \qquad
 n_3^0(x) = 7. \,\mathbf{1}_{[-25,-15]}(x).
 $$
 We observe that the \textit{Wolbachia}-infected population $n_3$ replaces the wild population $n_3$ in the region $\{x<0\}$. In Figure \ref{fig:Wolb2}, the only change is in the choice of the initial data for the \textit{Wolbachia}-infected population~:
 $$
 n_3^0(x) = 7 .\,\mathbf{1}_{[-20,-15]}(x).
 $$
 We observe that this initial data is not large enough to guarantee the success of the replacement strategy, as illustrated for instance in the work \cite{Zubelli}. 

Then, we take the following numerical values $K_1^F = 10$, $K_1^U = 4$, $K_2^F=1$, and $K_2^F = 4.5$. In this case, we have $\gamma_{13}^U>0>\gamma_{12}^U$. We display in Figure \ref{fig:Wolb3} the numerical results obtained with the following initial data~:
$$
n_1^0(x) = 2. \,\mathbf{1}_{[10,20]}(x), \qquad
n_2^0(x) = 3.5 \,\mathbf{1}_{x<0}(x), \qquad
n_3^0(x) = 3. \,\mathbf{1}_{[-25,-20]}(x).
$$
With this choice of initial data, the initial value of $n_3^0$ is not large enough to guarantee the replacement of the wild population of mosquitoes $n_2$ by the \textit{Wolbachia}-infected population $n_3$. Then, the population $n_3$ goes to extinction and finally we are dealing with the interaction between population $n_1$ and $n_2$ as in the previous section. Each of these populations establish in its respective region. 

Then we consider another choice of initial condition for $n_3^0$ which is now given by
$$
n_3^0(x) = 3. \,\mathbf{1}_{[-25,-10]}(x).
$$
In this case, the initial data is large enough to guarantee the success of the population replacement. The numerical results are displayed in Figure \ref{fig:Wolb4}. We observe the replacement of the wild population $n_2$ in the region $\{x<0\}$ by the \textit{Wolbachia}-infected population $n_3$. However, once the replacement is done, the population $n_3$ since to be less competitive than the population $n_1$ in the region $\{x<0\}$. Then we observe the invasion of the population $n_1$ at the expense of population $n_3$.
\begin{figure}[h!]
\centering\includegraphics[width=0.7\linewidth,height=5.7cm]{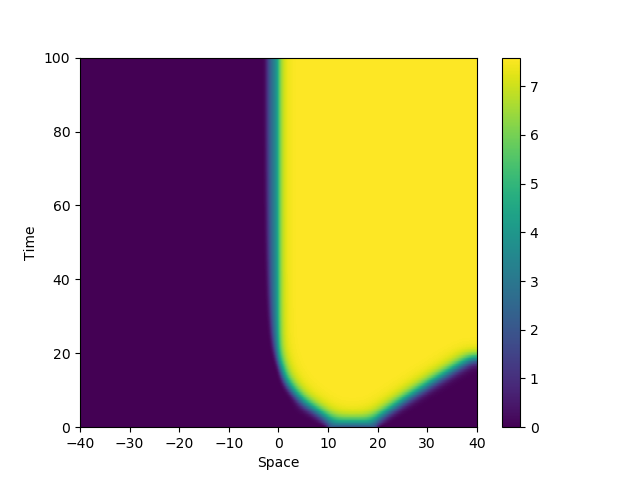}
\includegraphics[width=0.7\linewidth,height=5.7cm]{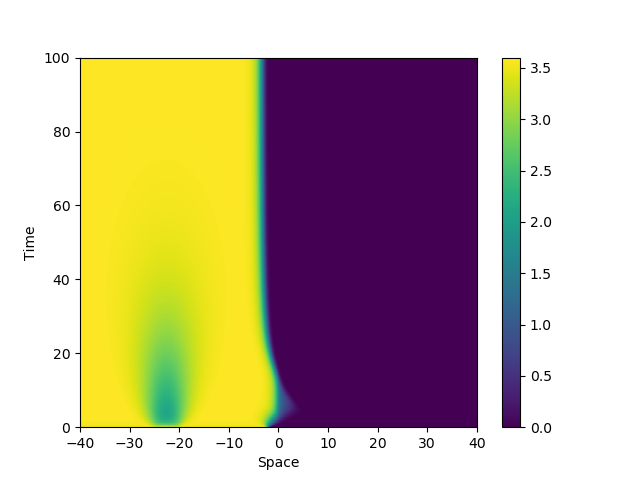}   
\includegraphics[width=0.7\linewidth,height=5.7cm]{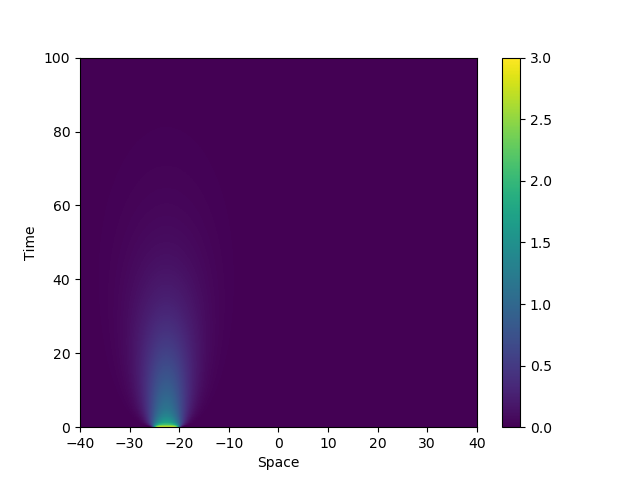}   
\caption{Time dynamics of functions $n_1$ (top), $n_2$ (middle), and $n_3$ (bottom) solution of system \eqref{systW} in the situation where $\gamma_{13}^U>0>\gamma_{12}^U$. In this case the initial condition for the Wolbachia-infected population $n_3$ is not large enough to succeed in replacing the wild population $n_2$ in the region $\{x<0\}$, while the population $n_1$ invades the domain $\{x>0\}$.}
\label{fig:Wolb3}
\end{figure}
\begin{figure}[h!]
\centering\includegraphics[width=0.7\linewidth,height=5.7cm]{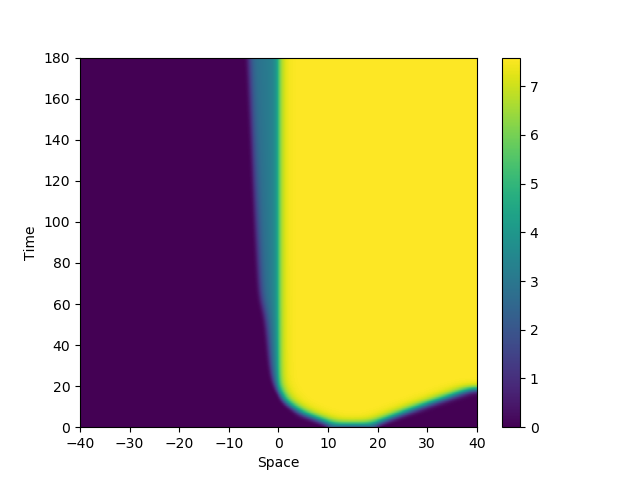}
\includegraphics[width=0.7\linewidth,height=5.7cm]{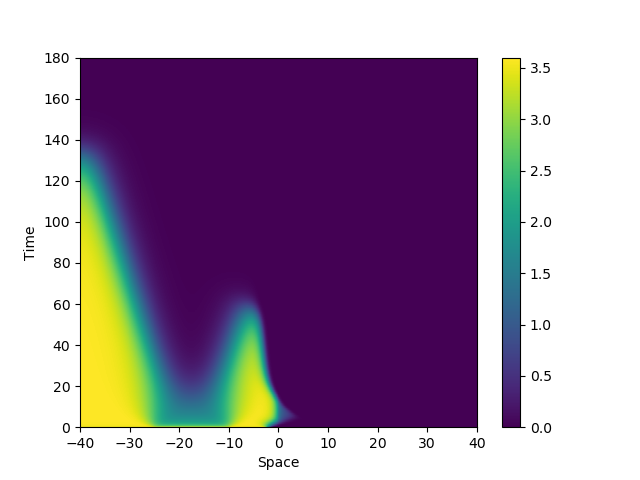}   
\includegraphics[width=0.7\linewidth,height=5.7cm]{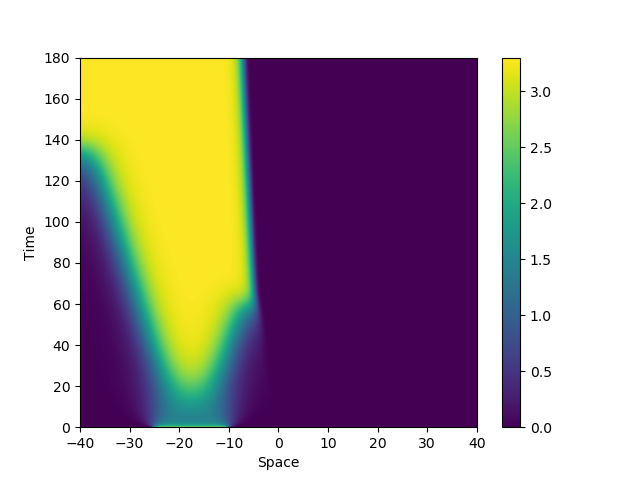}   
\caption{Time dynamics of functions $n_1$ (top), $n_2$ (middle), and $n_3$ (bottom) solution of system \eqref{systW} in the situation where $\gamma_{13}^U>0>\gamma_{12}^U$. In this case the initial condition for the Wolbachia-infected population $n_3$ is large enough to succeed in replacing the wild population $n_2$ in the region $\{x<0\}$. However after the replacement the population $n_1$ invades also the domain $\{x<0\}$.}
\label{fig:Wolb4}
\end{figure} 

\newpage
\textbf{Acknowledgements.}
This work has been initiated while N.V. was invited in the University of Tunis El Manar. N.V.
acknowledges warmly this institution for this opportunity and for the welcome.
N.V. acknowledges support from the STIC AmSud project BIO-CIVIP 23-STIC-02. \\
S.B. was invited in the Sorbonne University Paris North. S.B acknowledges warmly this institution for this opportunity and for the welcome, and the University of Tunis El Manar for the funding provided for my stays in France.


\end{document}